\documentclass[12pt]{amsart} 
\usepackage{amssymb,mathrsfs, longtable}%,stmaryrd,mathcomp,amsfonts, array}
\usepackage[matrix,arrow]{xy}%[matrix,arrow,curve]{xy}
\usepackage{url}

\makeatletter \@addtoreset{equation}{section} \makeatother

\newcommand{\down}[1]{\left\lfloor #1\right\rfloor}
\newcommand{\muu}{{\boldsymbol{\mu}}}

\newcommand{\PP}{{\mathbb P}}
\newcommand{\ZZ}{{\mathbb Z}}
\newcommand{\RR}{{\mathbb R}}
\newcommand{\FF}{{\mathbb F}}
\newcommand{\CC}{{\mathbb C}}
\newcommand{\NN}{{\mathbb Z}_{>0}}
\newcommand{\QQ}{{\mathbb Q}}
\newcommand{\OOO}{{\mathscr{O}}} 
\newcommand{\EEE}{{\mathscr{E}}} 
\newcommand{\III}{{\mathscr{I}}} 
\newcommand{\HHH}{{\mathscr{H}}} 
\newcommand{\MMM}{{\mathscr{M}}} 
\newcommand{\LLL}{{\mathscr{L}}} 
 
\newcommand{\TTT}{{\mathscr{T}}} 

\newcommand{\Cr}{\operatorname{Cr}}
\newcommand{\B}{{\mathbf B}}
\newcommand{\A}{{\mathfrak{A}}}
\newcommand{\Sy}{{\mathfrak{S}}}
\newcommand{\Or}{{\operatorname{O}}}
\newcommand{\PSL}{{\operatorname{PSL}}}
\newcommand{\SU}{{\operatorname{SU}}}
\newcommand{\PGL}{{\operatorname{PGL}}}
\newcommand{\PSp}{{\operatorname{PSp}}}
\newcommand{\M}{{\operatorname{M}}}

\newcommand{\GL}{{\operatorname{GL}}}
\newcommand{\SL}{{\operatorname{SL}}}
\newcommand{\Sp}{{\operatorname{Sp}}}
\newcommand{\SO}{{\operatorname{SO}}}

\newcommand{\mt}[1]{\operatorname{#1}}
\newcommand{\mult}{\operatorname{mult}}
\newcommand{\ord}{\operatorname{ord}}

\newcommand{\Ho}{H}
\newcommand{\h}{h}

\newcommand{\Bs}{\operatorname{Bs}}
\newcommand{\Cl}{\operatorname{Cl}}
\newcommand{\Sing}{\operatorname{Sing}}
\newcommand{\Pic}{\operatorname{Pic}}
\newcommand{\Aut}{\operatorname{Aut}}
\newcommand{\Bir}{\operatorname{Bir}}
\newcommand{\rk}{\operatorname{rk}}
\newcommand{\Gr}{\operatorname{Gr}}
\newcommand{\LGr}{\operatorname{LGr}}
\newcommand{\comment}[1]{}

 \newcommand{\Supp}{\operatorname{Supp}}
\newcommand{\Eu}{\operatorname{Eu}}

 \renewcommand{\emptyset}{\varnothing}
\newcommand{\xref}[1]{{\rm \ref{#1}}}

\newtheorem{theorem}[equation]{Theorem}

\newtheorem{lemma}[equation]{Lemma}
\newtheorem{corollary}[equation]{Corollary}

\newtheorem{claim}{Claim}[equation]

\theoremstyle{definition}
\newtheorem{definition}[equation]{Definition}
\newtheorem{remark}[equation]{Remark}
\newtheorem{example}[equation]{Example}
\newtheorem{case}[equation]{}

\newtheorem{mainassumption}[equation]{Assumption}

\newenvironment{subequation}{\begin{equation*}}{\refstepcounter{claim}\leqno (\theclaim)\end{equation*}}

\author{Yuri Prokhorov}
\thanks{
The author was partially supported by 
RFBR (grant No. 08-01-00395-a), 
Leading Scientific Schools (grants No. 1983.2008.1, 1987.2008.1), 
and 
AG Laboratory HSE (RF government grant ag. 11.G34.31.0023)}
\address{
Department of Higher Algebra, Faculty of Mathematics and Mechanics, 
Moscow State Lomonosov University, Vorobievy Gory, Moscow, 119 991, RUSSIA
\newline\indent
Laboratory of Algebraic Geometry, SU-HSE, 
7 Vavilova Str., Moscow, 117312, RUSSIA
}
\email{prokhoro@gmail.com} 
\title{Simple finite subgroups of the Cremona group of rank 3}
\date{}

\begin{document}
\maketitle
\begin{abstract}
We classify all finite simple subgroups of the Cremona group 
$\Cr_3(\CC)$. 
\end{abstract}

\section{Introduction}
Let $\Bbbk$ be a field.
The Cremona group $\Cr_d(\Bbbk)$ is the group 
of birational automorphisms 
of the projective space $\PP^d_{\Bbbk}$, or,
equivalently, the group of $\Bbbk$-automorphisms of 
the rational function field $\Bbbk(t_1,\dots,t_d)$.
It is well-known that $\Cr_1(\Bbbk)=\PGL_2(\Bbbk)$.
For $d\ge 2$, the structure of $\Cr_d(\Bbbk)$ and its subgroups 
is very complicated.
For example, the classification of finite subgroups in 
$\Cr_2(\CC)$ is an old classical problem.
Recently this classification almost has been 
completed by Dolgachev and Iskovskikh \cite{Dolgachev-Iskovskikh}.
The following is a consequence of the list in \cite{Dolgachev-Iskovskikh}.

\begin{theorem}[\cite{Dolgachev-Iskovskikh}]
\label{theorem-main-DI}
Let $G\subset \Cr_2(\CC)$ be a non-abelian simple finite subgroup.
Then $G$ is isomorphic to one of the following groups:
\begin{equation}
\label{eq-main-cr2}
\A_5,\quad \A_6, \quad\PSL_2(7).
\end{equation}
\end{theorem}
\par\noindent
However, the methods and results of \cite{Dolgachev-Iskovskikh} show that 
one cannot expect a reasonable classification of all finite 
subgroups of Cremona groups of higher rank.
In this paper we restrict ourselves with the case 
of simple finite subgroups of $\Cr_3(\CC)$.
Our main result is the following:

\begin{theorem}
\label{theorem-main-1}
Let $G\subset \Cr_3(\CC)$ be a non-abelian simple finite subgroup.
Then $G$ is isomorphic to one of the following groups:
\begin{equation}
\label{eq-main}
\A_5,\quad \A_6,\quad \A_7, \quad\PSL_2(7),\quad \SL_2(8),\quad \PSp_4(3).
\end{equation}
All the possibilities occur.
\end{theorem}
 
\par\noindent
In particular, we give the affirmative answer to a question of J.-P. Serre
\cite[Question 6.0]{Serre2009}: there are a lot of finite groups which 
do not admit any embeddings into $\Cr_3(\CC)$.
More generally we classify simple finite subgroups in 
the group of birational automorphisms 
of an arbitrary three-dimensional rationally connected variety and in many cases
we determine all birational models of the action:

\begin{theorem}
\label{theorem-main-2}
Let $X$ be a rationally connected threefold and let $G\subset \Bir (X)$ 
be a non-abelian simple finite subgroup.
Then $G$ is isomorphic either to $\PSL_2(11)$ or 
to one of the groups in the list \eqref{eq-main}.
All the possibilities occur. Furthermore, if $G$ does not admit any embeddings 
into $\Cr_2(\CC)$ \textup(see Theorem \xref{theorem-main-DI}\textup), then 
$G$ is conjugate to one of the following:
\begin{enumerate}
\item
$\A_7$ acting on some special smooth intersection of a quadric and a cubic
$X_6'\subset \PP^5$ \textup(see Example \xref{example-V6}\textup),

\item
$\A_7$ acting on $\PP^3$
\textup(see Theorem \xref{theorem-Blichfeldt}\textup),

\item
$\PSp_4(3)$ acting on $\PP^3$
\textup(see Theorem \xref{theorem-Blichfeldt}\textup),
\item
$\PSp_4(3)$ acting on the Burkhardt quartic
$X_4^{\mathrm b}\subset \PP^4$ 
\textup(see Example \xref{Burkhardt-quartic}\textup),
\item
$\SL_2(8)$ acting on some smooth Fano threefold $X_{12}^{\mathrm m}\subset \PP^8$
of genus $7$ 
\textup(see Example \xref{example-genus=7}\textup),

\item
$\PSL_2(11)$ acting on the Klein cubic $X^{\mathrm k}_3\subset \PP^4$
\textup(see Example \xref{example-Klein-cubic}\textup),
\item
$\PSL_2(11)$ acting on some smooth Fano threefold $X_{14}^{\mathrm a}\subset \PP^9$ 
of genus $8$
\textup(see Example \xref{example-V14}\textup).
\end{enumerate} 
Moreover, any equivariant action of $G$ on a Fano-Mori fiber space is isomorphic to one of the above cases. 
\end{theorem}
\par\noindent 
However we should mention that in contrast with \cite{Dolgachev-Iskovskikh} 
we do not describe  \textit{actions} of groups $\A_5$, $\A_6$ and $\PSL_2(7)$.
We also do not answer to the question about conjugacy 
groups (iii)-(iv), (vi)-(vii), and (i)-(ii).
\footnote{It was recently proved 
that $X^{\mathrm k}_3$ and $X_{14}^{\mathrm a}$ are not 
birationally $G$-isomorphic (see [Cheltsov I. and Shramov C. arXiv:0909.0918]).
I.~Cheltsov also pointed out to me that non-conjugacy of 
the actions of $\PSp_4(3)$ on the Burkhardt quartic $X_4^{\mathrm b}$
and $\PP^3$ follows from results of M. Mella and C.~Shramov
(see [Mella M.  Math. Ann. (2004) \textbf{330}, 107--126], 
[Shramov C. arXiv:0803.4348]).}

\begin{remark}
The cooresponding varieties in (ii)-(v) of the above theorem 
are rational. Hence these actions define embeddings of $G$ into $\Cr_3(\CC)$.
Varieties $X^{\mathrm k}_3$ and $X_{14}'$ are birationally equivalent 
and non-rational (see Remark \ref{remark-cubic} and \cite{Clemens-Griffiths}).
It is known that a general intersection of a quadric and a cubic is non-rational.
As far as I know the non-rationality of any smooth threefold 
in this family is still not proved. 
\end{remark}

\begin{remark}
(i) 
The orders of the above groups are as follows:
\[
\begin{array}{|c|ccccccc|}
\hline
&&&&&&&
\\[-5pt]
G&\A_5& \A_6& \A_7&\PSL_2(7)&\SL_2(8)& \PSp_4(3)& \PSL_2(11)
\\[5pt]
\hline
&&&&&&&
\\[-10pt]
|G|&60& 360 & 2520 & 168&504 &25920 &660
\\[5pt]
\hline
\end{array}
\]

(ii) 
There are well-known isomorphisms
$\PSp_4 (3)\simeq \SU_4(2)\simeq \Or_5(3)'$, 
$\A_5\simeq\SL_2(4)\simeq \PSL_2(5)$,
$\PSL_2(7)\simeq \GL_3(2)$, and
$\A_6\simeq\PSL_2(9) $ (see, e.g., \cite{atlas}).
\end{remark}

\par
The idea of the proof is quite standard.
It follows the classical ideas (cf. \cite{Dolgachev-Iskovskikh})
but has much more technical difficulties. Here is an outline 
of our approach.

By running the equivariant Minimal Model Program we may 
assume that our group $G$ acts on a Mori-Fano fiber space
$X/Z$. 
Here $Z$ is either a point, a rational curve or a rational surface 
(because a rationally connected surface over $\CC$ must be rational).
Since the group is simple and because
$G$ does not admit any embeddings into $\Cr_2 (\CC)$, we may assume that $Z$ is a point.
The latter means that $X$ is a $G\QQ$-Fano threefold.

\begin{definition}\label{definition-GQ-Fano}
A \emph{$G$-variety} is a variety $X$ provided with 
a biregular action of a finite group $G$.
We say that a normal $G$-variety 
$X$ is \emph{$G\QQ$-factorial} if any $G$-invariant Weil 
divisor on $X$ is $\QQ$-Cartier. 
A projective normal $G$-variety $X$ is called \emph{$G\QQ$-Fano}
if it is $G\QQ$-factorial, has at worst terminal 
singularities, $-K_X$ is ample, and $\rk \Pic (X)^G=1$. 
\end{definition}
Thus the $G$-equivariant Minimal Model Program reduces 
our problem to the classification of finite simple subgroups
in automorphism groups of $G\QQ$-Fano threefolds.
Smooth Fano threefolds are completely classified by 
Iskovskikh \cite{Iskovskikh-1980-Anticanonical} and Mori--Mukai
\cite{Mori1981-82}.
To study the singular case we use estimates for 
the number of singular points and analyze the action 
of $G$ on the singular set.

The structure of the paper is as follows.
In \S \ref{section-examples} we collect some examples
and show that all the cases in our list really occur.
Reduction to the case of $G\QQ$-Fano threefolds
is explained in \S \ref{section-Main-reduction}.
In 
\S \ref{section-Gorenstein} and \S \ref{section-non-Gorenstein}
we study the cases where $X$ is Gorenstein and non-Gorenstein, respectively.
 
\subsection*{Conventions.}
All varieties are defined over the complex number field $\CC$.
$\Sy_n$ and $\A_n$ denote the symmetric and the 
alternating groups, respectively. 
For linear groups over a field $\Bbbk$ we use the standard notations 
$\GL_n(\Bbbk)$, $\SO_n(\Bbbk)$, $\Sp_n(\Bbbk)$ etc.
If the field $\Bbbk$ is finite and contains $q$ elements, then, for short, the
above groups are denoted by $\GL_n(q)$, $\SO_n(q)$, $\Sp_n(q)$ etc.
For a group $G$, we denote by $Z(G)$ and $[G,G]$ its center and derived subgroup,
respectively. If the group $G$ acts on a set $\Omega$, then, for an element 
$P\in \Omega$, its stabilizer is denoted by $G_P$. 
All simple groups are supposed to be non-abelian.
The Picard number of a variety $X$ is denoted by $\rho(X)$.
For a normal variety $X$, $\Cl(X)$ is the Weil divisor class group.
Note that there is a difference between conjugate/non-conjugate
\textit{embeddings} $G \hookrightarrow\Cr_n(\Bbbk)$ and conjugate/non-conjugate 
\textit{subgroups} $G\subset \Cr_n(\Bbbk)$.
In this paper we discuss \textit{subgroups} $G\subset \Cr_n(\Bbbk)$.

\textbf{Acknowledgements.} 
This paper is dedicated to the memory of my teacher V. I. Iskovskikh.
In fact, my interest to Cremona groups was inspired by 
his works.
I would like to acknowledge discussions that I had on this subject 
with  I. Cheltsov, A.~Kuznetsov, and V. Popov.
I am grateful to D. Markushevich for pointing out the reference
\cite{AdlerRamanan1996}.
I am also would like to thank V.~Loginov, A.~A.~Mikhalev, and I.~Chubarov
for useful advises related to the group theory.
The work was conceived during my visit the University of Warwick
in 2008. Part of the work was done at MPIM, Bonn and 
POSTECH, Pohang, Korea.
I would like to thank these institutions for hospitality.
Finally I am  extremely
grateful to the referee, for careful reading and
numerous comments and suggestions that 
helped me to improve the presentation.

\section{Examples}\label{section-examples}
In this section we collect examples.

First of all, the group $\A_5$ acts on $\PP^1$ and $\PP^2$.
This gives a lot of embeddings into $\Cr_3(\CC)$ (by different actions on 
$\PP^1\times \PP^1\times \PP^1$ and $\PP^2\times \PP^1$).
The groups $\A_6$ and $\PSL_2(7)$ admit embeddings into 
$\Cr_2(\CC)$, so they are also can be embedded to $\Cr_3(\CC)$.

\begin{example}
Consider the embedding of $\A_5\subset \PGL_2(\CC)$
as a binary icosahedron group.
Let $H\subset \PGL_2(\CC)$ be another finite subgroup. 
Then there is an action of $\A_5$ on a rational homogeneous variety $\PGL_2(\CC)/H$.
This gives a series of embeddings of $\A_5$ into $\Cr_3(\CC)$.
\end{example}

Trivial examples also provide subgroups of $\PGL_4(\CC)$ 
(see Theorem \ref{theorem-Blichfeldt}):
$\A_5$, $\A_6$, $\A_7$, $\PSL_2(7)$, $\PSp_4 (3)$.
In the examples below we show that a finite simple group acts 
on a (possibly singular) Fano thereefold.
According to \cite{Zhang-Qi-2006} Fano varieties with log terminal singularities are rationally
connected, so our constructions give embeddings of 
a finite simple group into the automorphism group
of some rationally connected variety. 

\begin{example}
The group $\A_5$ acts on 
the smooth cubic $\{\sum_{i=1}^4 x_i^3=0\}\subset \PP^4$
and on the smooth quartic $\{\sum_{i=1}^4 x_i^4=0\}\subset \PP^4$.
These varieties are not rational \cite{Clemens-Griffiths}, \cite{Iskovskikh1971}. 
\end{example}

\begin{example}
The \textit{Segre cubic} $X_3^{\mathrm s}$ is a subvariety in $\PP^5$
given by the equations $\sum x_i=\sum x_i^3=0$.
This cubic has 10 nodes, it is obviously rational, 
and $\Aut X_3^{\mathrm s}\simeq \Sy_6$.
In particular, alternating groups $\A_5$ and $\A_6$ act on $X_3^{\mathrm s}$.
Since $\A_5$ can be embedded into $\A_6$ in two ways, 
this construction gives two embeddings of $\A_5$ into $\Cr_3(\CC)$.
We do not know if they are conjugate or not.
\end{example}

\begin{example}
Assume that $G$ acts on $\CC^5$ so that there are
(irreducible) invariants $\phi_2$ and $\phi_3$ of degree $2$ and $3$,
respectively.
Let $Y\subset \PP^4$ be a (possibly singular) cubic hypersurface
given by $\phi_3=0$ and let 
$R\subset Y$ be the surface given by $\phi_2=\phi_3=0$.
Then $R\in |{-}K_Y|$. Consider the double cover $X\to Y$ ramified 
along $R$. Then $X$ is a Fano threefold.
It can be realized as an intersection of a cubic and 
quadratic cone in $\PP^5$. The action of $G$ lifts to $X$.
There are two interesting cases (cf. \cite{Mukai1988}):
\begin{enumerate}
\item[(a)]
$Y=\{\sum_{i=0}^5 x_i=\sum_{i=0}^5 x_i^3=0 \}\subset \PP^5$
is the Segre cubic, and $R$ is cut out by the equation $\sum_{i=0}^5 x_i^2=0$, 
$G=\A_6$;
\item[(b)]
$Y=\{\sum_{i=0}^4 x_i=\sum_{i=0}^4 x_i^3=0 \}\subset \PP^5$
is a cubic cone, 
and $R$ is cut out by the equation $\sum_{i=0}^5 x_i^2=0$, 
$G=\A_5$.
\end{enumerate}
\end{example}

\begin{example}
\label{example-V6}
Consider the subvariety in $X_6'\subset\PP^6$ given by the equations
$\sigma_1=\sigma_2=\sigma_3=0$, where $\sigma_i$ are symmetric 
polynomials in $x_1,\dots,x_7$.
Then $X_6'$ is Fano smooth threefold, an intersection of 
a quadric and a cubic in $\PP^5$.
The alternating group $\A_7$ naturally acts on $X_6'$.
A general variety $X_6$ in this family is not rational 
\cite{Iskovskikh-Pukhlikov-1996}. The (non-)rationality of $X_6'$
is not known\footnote{The non-rationality of $X_6'$  was proved recently by A. Beauville
[Non-rationality of the symmetric sextic Fano threefold, arXiv:1102.1255].}. 
\end{example}

\begin{example}
\label{example-Klein-cubic}
The automorphism group of the cubic $X^{\mathrm k}_3\subset \PP^4$ given by the equation 
\begin{equation}
\label{equation-Klein-cubic}
x_1^2x_2+x_2^2x_3+x_3^2x_4+x_4^2x_5+x_5^2x_1=0
\end{equation}
is isomorphic to $\PSL_2(11)$. This was discovered by F. Klein, see
\cite{Adler1978} for a complete modern proof.
\footnote{As pointed out by 
J. Ellenberg, the existence of this action can also be seen from the fact that
the cubic $X^{\mathrm k}_3$ is birational to 
$\mathcal A_{11}^{lev}$, the moduli space of abelian surfaces with 
$(1,11)$-polarization and canonical level structure, see
[M. Gross and S. Popescu. 
The moduli space of $(1,11)$-polarized abelian surfaces is unirational.
\textit{Compositio Math.}, 126:1--23, 2001]. 
} 
\end{example}

\begin{example}
\label{Burkhardt-quartic}
The \textit{Burkhardt quartic} $X_4^{\mathrm b}$ is a subvariety in $\PP^5$
given by $\sigma_1=\sigma_4=0$, where 
$\sigma_i$ is $i$-th symmetric function in $x_1,\dots,x_6$.
The automorphism group of $X_4^{\mathrm b}$ is 
isomorphic to $\PSp_4(3)$, see \cite{Shephard-Todd}.
\footnote{Similar to the previous example the existence of this action follows  also 
from an interpretation of $X_4^{\mathrm b}$ as a moduli space, see, e.g., 
[B. van Geemen.
Projective models of Picard modular varieties. in
\textit{Classification of irregular varieties} (Trento, 1990), 68--99,
Lecture Notes in Math., \textbf{1515}, Springer, Berlin, 1992]}
\end{example}

\begin{example}
\label{example-V14}
Let $W$ be a $5$-dimensional irreducible representation of $\tilde G:=\SL_2(11)$.
Consider
the following skew symmetric matrix whose entries are linear forms on $W$:
\[
A:=\begin{pmatrix}
0&x_4&x_5&x_1&x_2&x_3
\\
-x_4&0&0&x_3&-x_1&0
\\
-x_5&0&0&0&x_4&-x_2
\\
-x_1&-x_3&0&0&0&x_5
\\
-x_2&x_1&-x_4&0&0&0
\\
-x_3&0&x_2&-x_5&0&0
\end{pmatrix}
\]
The matrix $A$ can be regarded as a non-trivial $G$-equivariant linear map from $W$ to 
$\wedge^2 V$, where $V$ is a $6$-dimensional 
irreducible representation of $\tilde G$, see 
\cite[\S 47]{AdlerRamanan1996}.
Thus the representation $\wedge^2 V$ is decomposed as
$\wedge^2 V=W\oplus W^{\perp}$, where $\dim W^{\perp}=10$.
Let $X_{14}^{\mathrm a}:=\PP(W^{\perp})\cap \Gr(2,V)\subset \PP(\wedge^2 V)$.

It is easy to check that $\rk A(w)\ge 4$ for 
any $w\in W$, $w\neq 0$. Thus $A$ is a regular net
of skew forms in the sense of \cite{Kuznetsov2004-e}.
The Pfaffian of $A$ defines a cubic hypersurface
$X_3\subset \PP^4$.
This hypersurface $X_3$ is given by the equation \eqref{equation-Klein-cubic}
because the action of $\SL_2(11)$ on $\CC^5$ has only one invariant of degree
$3$ (see \cite{Adler1978}). So,
$X_3=X^{\mathrm k}_3$.
Hence it is smooth and so is $X_{14}^{\mathrm a}$ by \cite[Prop. A.4]{Kuznetsov2004-e}.
Then by the adjunction formula $X_{14}^{\mathrm a}$ is 
a Fano threefold of Picard number one and genus $8$
\cite{Iskovskikh-Prokhorov-1999}.
By construction $X_{14}^{\mathrm a}$ admits a non-trivial action of $G$.
\end{example} 

\begin{remark}
\label{remark-cubic}
It turns out that $X_{14}^{\mathrm a}$ and $X^{\mathrm k}_3$ are birationally equivalent
(and not rational \cite{Clemens-Griffiths}), 
so our construction does not give an embedding of $G$ into
$\Cr_3(\CC)$. 
The birational equivalence of $X_{14}^{\mathrm a}$ and $X^{\mathrm k}_3$ can be seen from the following construction
\cite{Puts1982}.
Given a smooth section $X_{14}=\Gr(2,6)\cap \PP^9$, let 
$Y\subset \PP^{5}$ be the variety swept out by 
lines representing points of $X_{14}\subset \Gr(2,6)$.
Then $Y$ is a singular quartic fourfold. 
It is called the \textit{Palatini quartic}
of $X_{14}$.
In our case $X_{14}=X_{14}^{\mathrm a}$
the equation of $Y$ is as follows \cite[Cor. 50.2]{AdlerRamanan1996}:
\begin{multline*}
x_0^4+
x_0(x_1^2x_2+x_2^2x_3+x_3^2x_4+x_4^2x_5+x_5^2x_1) +
\\
+x_1^2x_3x_5+x_2^2x_4x_1+x_3^2x_5x_2+x_4^2x_1x_3+
x_5^2x_2x_4=0.
\end{multline*}
Let $H$ be a general hyperplane section of $Y$.
Then $H$ is a quartic threefold with 25 singular points
which is birational to both $X_{14}^{\mathrm a}$ and $X^{\mathrm k}_3$ \cite{Puts1982},
\cite[Theorem 13.11]{Clemens-Griffiths}.
Note however that our birational construction is not $G$-equivariant.
We do not know whether our two embeddings of $G$ into 
$\Bir (X_{14}^{\mathrm a}) \simeq \Bir (X^{\mathrm k}_3)$
are conjugate or not.
\end{remark}

\begin{example}[\cite{Mukai1992b}]
\label{example-genus=7}
There is a curve $C^{\mathrm m}$ of genus $7$ for which 
the Hurwitz bound of the automorphism group is achieved
\cite{Macbeath1965}.
In fact, $\Aut C^{\mathrm m}\simeq \SL_2(8)$.
The ``dual'' Fano threefold of genus $7$
has the same automorphism group.
The construction due to S. Mukai \cite{Mukai1992b, mukai-1995-1}
is as follows.
Let $Q\subset \PP^8$ be a smooth quadric.
All $3$-dimensional projective subspaces of $\PP^8$ contained in $Q$
are parameterized by a smooth irreducible $\SO_9(\CC)$-homogeneous variety
$\LGr(4,9)$, so-called, Lagrangian Grassmannian.
In fact, $\LGr(4,9)$ is 
a Fano manifold of dimension $10$
and Fano index $8$ with $\rho(\LGr(4,9))=1$.
The positive generator of $\Pic (\LGr(4,9))\simeq \ZZ$
determines an embedding $\LGr(4,9)\hookrightarrow \PP^{15}$.
In fact, this embedding is given by the spinor coordinates on 
$\LGr(4,9)$. It is known 
that any smooth Fano threefold $X_{12}^{\mathrm m}$ of genus $7$ with 
$\rho(X_{12}^{\mathrm m})=1$
is isomorphic to a section of $\LGr(4,9)\subset \PP^{15}$ by a subspace of dimension 
$8$ \cite{Mukai-1988}. Similarly, any canonical 
curve $C$ of genus $7$ is isomorphic to a section of $\LGr(4,9)\subset \PP^{15}$
by a subspace of dimension 
$6$ if and only if $C$ has no $g^1_4$
\cite{mukai-1995-1}.
The group $\SL_2(8)$ has a $9$-dimensional representation $U$
and there is an invariant quadric $Q\subset \PP(U)$.
Hence $\SL_2(8)$ naturally acts on $\LGr(4,9)$.
This action lifts to $\PP^{15}$ so that there are two 
invariant subspaces $\Pi_1$ and $\Pi_2$ of dimension $6$ and $8$,
respectively. The intersections $\LGr(4,9)\cap \Pi_1$ and $\LGr(4,9)\cap \Pi_2$
are our curve $C^{\mathrm m}$ \cite[Table 1]{mukai-1995-1} and a smooth Fano threefold of genus $7$ with $\rho=1$
(see \cite[Lemma 3.2]{Iliev2004}). 
Recall that any smooth Fano threefold of genus $7$ with $\rho=1$ is rational
(see, e.g., \cite{Iskovskikh-Prokhorov-1999}).
Therefore, the above construction provides an embedding of
$\SL_2(8)$ into $\Cr_3(\CC)$.
\end{example}

\section{Finite linear and permutation groups}
\begin{case}
\textbf{Finite linear groups.}
Let $V$ be a vector space. 
An irreducible subgroup $G\subset \GL(V)$ is said to be 
\textit{imprimitive} if 
there exists a non-trivial decomposition $V=\oplus V_i$
such that $G$ permutes subspaces $V_i$.
In this case $G$ contains a non-trivial reducible normal subgroup $N$
such that $g V_i= V_i$ for all $g\in N$ and all $i$.
A group $G$ is said to be \textit{primitive} if it is irreducible and 
not imprimitive. Clearly, a simple linear group has to be primitive if it is irreducible.
\end{case}

All finite primitive linear groups of small degree 
are classified, see \cite{Blichfeldt}, \cite{Brauer-1967}, and \cite{Lindsey-1971-lg6}.
Basically we need only the list of the simple ones.

\begin{theorem}[{\cite{Blichfeldt}}]
\label{theorem-classification-3}
Let $G\subset \PGL_3(\CC)$ be a finite irreducible simple 
subgroup and let $\tilde G \subset \SL_3(\CC)$ be its preimage under
the natural map $\SL_3(\CC)\to \PGL_3(\CC)$ such that 
$[\tilde G,\tilde G]\supset Z(\tilde G)$. Then only
one of the following cases is possible:
\begin{enumerate}
\item 
the icosahedral group, $G\simeq \tilde G \simeq\A_5$;
\item
the Valentiner group, $G\simeq\A_6$, $Z(\tilde G)\simeq \muu_3$;
\item 
the Klein group, $G\simeq\tilde G \simeq\PSL_2(7)$;
\item
the Hessian group, $G\simeq (\muu_3)^2\rtimes \SL_2(3)$, $|G|=216$,
$Z(\tilde G)\simeq \muu_3$;
\item
subgroups of the Hessian group of index $3$ and $6$.
\end{enumerate}
\end{theorem}

\begin{theorem}[{\cite{Blichfeldt}}]
\label{theorem-Blichfeldt}
Let $G\subset \PGL_4(\CC)$ be a finite irreducible simple 
subgroup and let $\tilde G \subset \SL_4(\CC)$ be its preimage under
the natural map $\SL_4(\CC)\to \PGL_4(\CC)$ such that 
$[\tilde G,\tilde G]\supset Z(\tilde G)$. Then
one of the following cases is possible:
\begin{enumerate}
\item 
$G\simeq \tilde G\simeq\A_5$, 
\item 
$G\simeq \A_6$, $Z(\tilde G)\simeq \muu_2$,
\item 
$G\simeq \A_7$, $Z(\tilde G)\simeq \muu_2$,
\item
$G\simeq \A_5$, $\tilde G\simeq \SL_2(5)$, $Z(\tilde G)\simeq \muu_2$,
\item
$G\simeq \PSL_2(7)$, $\tilde G\simeq \SL_2(7)$, $Z(\tilde G)\simeq \muu_2$,
\item
$G\simeq\PSp_4 (3)$, 
$\tilde G=\Sp_4(3)$, $Z(\tilde G)\simeq \muu_2$.
\end{enumerate}
\end{theorem}

\begin{theorem}[{\cite{Brauer-1967}}]
\label{theorem-Brauer}
Let $G\subset \PGL_5(\CC)$ be a finite irreducible simple 
subgroup and let $\tilde G \subset \SL_5(\CC)$ be its preimage under
the natural map $\SL_5(\CC)\to \PGL_5(\CC)$ such that 
$[\tilde G,\tilde G]\supset Z(\tilde G)$. Then $G\simeq \tilde G$ and
only one of the following cases is possible:
\[
\A_5, 
\quad
\A_6, 
\quad
\PSL_2(11), 
\quad
\PSp_4 (3).
\]
\end{theorem}

\begin{theorem}[{\cite{Lindsey-1971-lg6}}]
\label{theorem-Lindsey}
Let $G\subset \GL_6(\CC)$ be a finite irreducible simple 
subgroup. Then $G$ is isomorphic to one of the following groups:
\[
\A_7,
\quad
\PSL_2(7),
\quad
\PSp_4(3),
\quad
\SU_3(3).
\]
\end{theorem}

\begin{lemma}
\label{lemma-quadric4}
Let $G$ be a finite simple group.
Assume that $G$ admits an embedding into $\operatorname{PSO}_n(\CC)$
with $n\le 6$ and does not admit any embeddings into $\Cr_2(\CC)$. Then $n=6$ and $G$ is 
isomorphic to $\A_7$ or $\PSp_4(3)$. 
\end{lemma}
\begin{proof}
We may assume that $G\subset \operatorname{PSO}_6(\CC)$ (we can use embeddings 
$\operatorname{PSO}_r(\CC)\subset \operatorname{PSO}_6(\CC))$ for $r<6$). Thus
$G$ acts faithfully on a smooth quadric $Q\subset \PP^5$.
It is well known that $Q$ contains two $3$-dimensional families 
$F_1$, $F_2$ of planes \cite[Ch. 6, \S 1]{Griffiths1978}. 
Regarding $Q$ as the Grassmann variety $\Gr(2,4)$ we see $F_1\simeq F_2\simeq \PP^3$
\cite[Ch. 6, \S 2]{Griffiths1978}. We get a non-trivial action of
$G$ on $\PP^3$. Now the assertion follows by Theorems \ref{theorem-Blichfeldt}.
\end{proof}

\subsection*{Transitive simple permutation groups}
Let $G$ be a group acting transitively on a finite set $\Omega$.
A nonempty subset $\Omega'\subset \Omega$ is called a \textit{block} for $G$
if for each $\delta\in G$ either $\delta(\Omega')=\Omega'$ or $\delta(\Omega')\cap \Omega'=\emptyset$. 
If $\Omega'\subset \Omega$ is a block for $G$, then for any $\delta\in G$
the image $\delta(\Omega')$ is also a block and 
the system of all such blocks forms a partition of $\Omega$.
Moreover, the setwise stabilizer $G_{\Omega'}$ acts on $\Omega'$ 
transitively.
The action of $G$ is said to be \textit{imprimitive} if there is a 
block $\Omega'\subset \Omega$ containing more than one element.
Otherwise the action is said to be \textit{primitive}.

Below we list all finite simple transitive permutation 
groups acting on $n\le 26$ symbols \cite{Dixon-Mortimer}.
\begin{theorem}
\label{theorem-transitive-groups}
Let $G$ be a finite transitive permutation 
group acting on the set $\Omega$ with $|\Omega|\le 26$. 
Assume that $G$ is simple and is not contained in the list 
\eqref{eq-main-cr2}. 
Then the action is primitive and we have one of the following cases:
%\par\medskip\noindent
\newpage
\setlongtables\renewcommand{\arraystretch}{1.3}
{\rm
\begin{longtable}{|p{12pt}|p{68pt}|p{60pt}|p{112pt}|p{74pt}|}
\hline
$|\Omega|$&$G$&$\Omega$&degrees of irreducible representations 
in the interval $[2,14]$&$G_P$ 
\\
\hline
\endhead
\hline
\endlastfoot
\hline
\endfoot
\hline
\multicolumn{5}{|c|}{primitive groups} 
\\
\hline
$n$& $\A_n$&standard &6, 10, 14 \ if $n=7$&$\A_{n-1}$
\\
&&& $7$, $14$ \qquad if $n=8$ &
\\
&&& $n-1$ \qquad if $n\ge 9$ &
\\
9& $\SL_2(8)$& $\PP^1(\FF_{8})$ & 7, 8, 9 & $(\muu_2)^{3}\rtimes \muu_7$
\\
11& $\PSL_2(11)$& &
 $5$, $10$, $11$, $12$
& $\A_5$ 
\\
11& $\M_{11}$&standard & 10, 11&$\M_{10}$
\\
12& $\M_{11}$&& 10, 11&$\PSL_2(11)$
\\
12& $\M_{12}$&standard& 11&$\M_{11}$
\\
12 &$\PSL_2(11)$ &$\PP^1(\FF_{11})$
&
 $5$, $10$, $11$, $12$
& $\muu_{11}\rtimes \muu_5$ 
\\
13& $\SL_3(3)$&$\PP^2(\FF_{3})$&
 $12$, $13$
& $(\muu_{3})^2\rtimes \GL_2(3)$ 
\\
14 &$\PSL_2(13)$&$\PP^1(\FF_{13})$&
 $7$, $12$, $13$, $14$
 & $\muu_{13}\rtimes \muu_6$
\\
15& $\A_7$ & & 6, 10, 14 & $\PSL_2(7)$
\\
15 &$\A_8\simeq\SL_4(2)$&$\PP^3(\FF_{2})$&
$7$, $14$ & $(\muu_2)^3 \rtimes \SL_3(2)$
\\
17&$\SL_2(16)$&$\PP^1(\FF_{16})$&$-$
&$(\muu_2)^4 \rtimes \muu_{15}$
\\
18&$\PSL_2(17)$&$\PP^1(\FF_{17})$& $9$ &
$\muu_{17} \rtimes \muu_{8}$
\\
20
&$\PSL_2(19)$&$\PP^1(\FF_{19})$&9&
$\muu_{19} \rtimes \muu_{9}$
\\
21& $\A_7$ & & 6, 10, 14 & $\Sy_5$
\\
21&
$\PSL_3(4)$&$\PP^2(\FF_{4})$&$-$&
$(\muu_{2})^4 \rtimes \SL_2(4)$
\\
22
&
$\M_{22}$&standard&$-$&$\PSL_3(4)$
\\
23&
$\M_{23}$&standard&$-$&$\M_{22}$
\\
24&
$\M_{24}$&standard&$-$&$\M_{23}$
\\
24
&$\PSL_2(23)$&$\PP^1(\FF_{23})$&
11&
$\muu_{23} \rtimes \muu_{11}$
\\
26
&
$\PSL_2(25)$&$\PP^1(\FF_{25})$&13&
$(\muu_{5})^2 \rtimes \muu_{12}$
\\
\hline
\multicolumn{5}{|c|}{imprimitive groups} 
\\
\hline
22&$\M_{11}$&&10, 11&$\A_6$
\\
26&$\SL_3(3)$&$\mathbb{A}^3(\FF_3)\setminus\{0\}$&12, 13&$(\muu_3)^2\rtimes \SL_2(3)$
\end{longtable}
}
\noindent
Here $\M_k$ denotes the Mathieu group, $G_P$ is the stabilizer of $P\in \Omega$ 
and $\PP^m(\FF_q)$ \textup(resp. $\mathbb{A}^m(\FF_q)$\textup) 
denotes the projective \textup(resp. affine\textup) space over the 
finite field $\FF_q$. 
\end{theorem}

All primitive permutation groups are taken from the book \cite{Dixon-Mortimer}.
Their irreducible representations can be found in \cite{atlas}. 
So we need to consider only imprimitive case.

\begin{proof}
[Proof of Theorem \xref{theorem-transitive-groups} in the imprimitive case]
If the group 
$G$ is imprimitive, then $G$ acts on the system of blocks $\Lambda$, where 
$|\Lambda|=|\Omega|/m\le 13$ and $m\ge 2$ is the number of elements in a block.
Then $m\le 3$, the action on $\Lambda$ is primitive, and for a block $\Omega'$,
the setwise stabilizer $G_{\Omega'}$ acts on $\Omega'$ 
transitively. This gives us two possibilities: $\M_{11}$ and $\SL_3(3)$.
\end{proof}

\begin{remark}
We will show that the group $\A_8$ cannot act 
non-trivially on a rationally connected threefold.
Hence the same holds for all $\A_n$ with $n\ge 9$.
Therefore, in order to prove Theorems \ref{theorem-main-1} and \ref{theorem-main-2} we should
not worry about groups $\A_n$ for  $n\ge 9$.
\end{remark}

\begin{corollary}
\label{corollary-transitive-groups}
In notation of Theorem \xref{theorem-transitive-groups} 
the stabilizer $G_P$ has a faithful representation of degree $\le 4$ 
only in the following cases:
\begin{enumerate}
 \item 
$G\simeq\PSL_2(11)$,\quad $|\Omega|=11$,\quad $G_P\simeq \A_5$;
 \item
$G\simeq\A_7$,\quad $|\Omega|=15$,\quad $G_P\simeq \PSL_2(7)$;
 \item 
$G\simeq\A_7$,\quad $|\Omega|=21$,\quad $G_P\simeq \Sy_5$.
\end{enumerate}
\end{corollary}
\begin{proof}
Clearly, in the above cases (i)-(iii) the group $G_P$ has a faithful
representation of degree $\le 4$. 

\par\smallskip\noindent
\textbf{Cases $G_P\simeq \PSL_3(4)$, $\A_{n-1}$ with $n\ge 7$, $\M_n$ with $n=10$, $11$, $22$, $23$. }
Then
$G_P$ has no faithful representations of degree $\le 4$ (see, e.g., \cite{atlas}
or Theorems \ref{theorem-classification-3}
 and \ref{theorem-Blichfeldt}).

In the remaining cases of Theorem \ref{theorem-transitive-groups}
the group $G_P$ is a semi-direct product $A\rtimes B$, where $A$ is abelian
and the action of $B$ on $A$ is faithfull.

\begin{claim}\label{claim-last}
\begin{enumerate}
\item 
$A$ is a maximal abelian normal subgroup of $G_P$. 
\item 
No non-trivial subgroups of $A$ are normal in $G_P$.
\item 
Any non-trivial normal subgroup of $G_P$ contains $A$.
\end{enumerate}
\end{claim}
\begin{proof}[Proof of the claim]
The statement of (i) is obvious because the action of $B$ on $A$ is faithfull.

(ii) Let $\{1\}\neq N\subsetneq A$ be a subgroup that is normal in $G_P$.
Then $A$ cannot be a cyclic group of prime order.
In cases $G_P\simeq$ $(\muu_2)^{3}\rtimes \muu_7$, $(\muu_2)^3 \rtimes \SL_3(2)$,
$(\muu_2)^4 \rtimes \muu_{15}$,
$(\muu_{2})^4 \rtimes \SL_2(4)$
the group $B$ transitively acts on $A\setminus \{1\}$.
So $N$ cannot be normal.
In the remaining cases $G_P\simeq$ 
$(\muu_{3})^2\rtimes \GL_2(3)$, 
$(\muu_{5})^2 \rtimes \muu_{12}$,
$(\muu_3)^2\rtimes \SL_2(3)$  the group $N$ must be a cyclic group of prime order
and one can see immediately that $N$ is not normal in $G_P$.

(iii) 
Let $N\subset G_P$ is a non-trivial normal subgroup.
By (ii) we may assume that $N\cap A=\{1\}$. 
Thus $A\times N$ is a normal subgroup of $G_P$.
Since the action of $B$ on $A$ is faithfull, this is impossible.
\end{proof}

Assume that $G_P$ has a faithful representation $V$ of degree $\le 4$.
Take $V$ so that its degree is minimal possible.
If $V$ is reducible, then $V=V_1\oplus V_2$, where both $V_i$ are non-trivial
non-faithful representations.
By Claim \ref{claim-last} kernels of these representations contain $A$.
So $V$ is not faithful, a contradiction.

Thus $V$ is irreducible.
Then the action of $G$ on $V$ is imprimitive (because $A$ contains an abelian normal subgroup)
and the induced action on eigenspaces $V_1,\dots, V_n$ of $A$ 
induces a transitive embedding of $B$ into 
$\Sy_n$ with $n\le 4$.
But in our cases 
$B$ is isomorphic to either $\GL_2(3)$, 
$\SL_3(2)$, $\SL_2(4)$, 
$\SL_2(3)$, or $B\simeq \muu_l$, with $l\ge 5$.
This group does not admit any embeddings into $\Sy_4$, a contradiction.
\end{proof}

\section{Main reduction}
\label{section-Main-reduction}
\begin{case}
\label{Terminal-singularities}
\textbf{Terminal singularities.} 
Here we list only some of the necessary results on three-dimensional 
terminal singularities.
For more complete information we refer to \cite{Reid-YPG1987}.
Let $(X, P)$ be a germ of a three-dimensional terminal singularity.
Then $(X, P)$ is isolated, i.e, $\Sing(X)=\{P\}$.
The \textit{index} of $(X, P)$ is the minimal positive integer
$r$ such that $rK_X$ is Cartier.
If $r=1$, then $(X, P)$ is Gorenstein. In this case $\dim T_{P,X}=4$,
$\mult (X, P)=2$, and $(X, P)$ is analytically isomorphic to a hypersurface 
singularity in $\CC^4$.
If $r>1$, then there is a cyclic, \'etale outside of $P$ 
cover $\pi : (X^{\sharp},P^{\sharp})\to (X,P)$ 
of degree $r$ such that $(X^{\sharp},P^{\sharp})$ is a Gorenstein 
terminal singularity (or a smooth point). This $\pi$ is called
the \textit{index-one cover} of $(X, P)$. 
If $(X^{\sharp},P^{\sharp})$ is smooth, then the point 
$(X, P)$ is analytically isomorphic to a 
quotient $\CC^3/\muu_r$, where the weights $(w_1,w_2,w_3)$ of the action of 
$\muu_r$ up to permutations satisfy the relations $w_1+w_2\equiv 0\mod r$ and $\gcd(w_i,r)=1$.
This point is called a \textit{cyclic quotient} singularity.

For any three-dimensional terminal singularity $(X, P)$ of index $r\ge 1$
there exists a one-parameter deformation
$\mathfrak X \to \Delta \ni 0$ over a small disk $\Delta \subset \CC$ 
such that the central fiber $\mathfrak X_0$ is
isomorphic to $X$ and the general fiber 
$\mathfrak X_\lambda$ has only cyclic quotient terminal singularities 
$P_{\lambda,k}$. 
Thus, one can associate with a fixed threefold $X$ with terminal
singularities a collection 
$\B =\{(\mathfrak X_\lambda, P_{\lambda,k})\}$ of cyclic quotient singularities. 
This collection is uniquely determined by the variety
$X$ and is called the \textit{basket} of singularities of $X$. 

If $(X,P)$ is a singularity of index one, then it is an 
isolated hypersurface singularity. Hence 
$X\setminus\{P\}$ is 
simply-connected and 
the (local) Weil divisor class group $\Cl(X)$ is torsion free.
If $(X, P)$ is of index $r> 1$, then the index one cover 
induces the topological universal cover 
$X^{\sharp}\setminus \{P^{\sharp}\}\to X\setminus \{P\}$. 
\end{case}

\begin{case}
\label{main-reduction}
\textbf{$G$-equivariant minimal model program.}
Let $X$ be a rationally connected three-dimensional algebraic variety
and let $G\subset \Bir(X)$ be a finite subgroup.
By shrinking $X$ we may assume that $G$ acts on $X$ biregularly.
The quotient $Y=X/G$ is quasiprojective, so 
there exists a projective completion $\hat Y\supset Y$.
Let $\hat X$ be the normalization of $\hat Y$ in the function field
$\CC(X)$. Then $\hat X$ is a projective variety birational to $X$ 
admitting a biregular action of $G$.
There is an equivariant resolution of singularities $\tilde X\to \hat X$,
see \cite{Abramovich-Wang}.
Run the $G$-equivariant minimal model program: $\tilde X\to \bar X$,
see \cite[0.3.14]{Mori-1988}.
Running this program we stay in the category of projective normal varieties 
with at worst terminal $G\QQ$-factorial singularities. 
Since $X$ is rationally connected, on the final step we get 
a Fano-Mori fibration $f: \bar X\to Z$. Here $\dim Z<\dim X$, $Z$ is normal,
$f$ has connected fibers, the anticanonical Weil
divisor $-K_{\bar X}$ is ample over $Z$, and the relative $G$-invariant Picard 
number $\rho(\bar X)^G$ is one.
Obviously, we have the following possibilities:
\begin{enumerate}
\item 
$Z$ is a rational surface and a general fiber $F=f^{-1}(y)$ is a conic;
\item
$Z\simeq \PP^1$ and a general fiber $F=f^{-1}(y)$ is a smooth del Pezzo surface;
\item
$Z$ is a point and $\bar X$ is a $G\QQ$-Fano threefold.
\end{enumerate}
Now we assume that $G$ is a simple group.
If $Z$ is not a point, then $G$ non-trivially acts either on the base $Z$
or on a general fiber. Both of them are rational varieties. Hence $G\subset \Cr_2(\CC)$
in this case. Thus we may assume that we are in the case (iii).
Replacing $X$ with $\bar X$ we may assume that our original $X$ is a $G\QQ$-Fano threefold. 

In some statements below this assumption 
will be weakened.
For example we will assume sometimes that $-K_X$ is just nef 
and big (not ample).
We need this for some technical reasons (see \S \ref{section-non-Gorenstein}).
\end{case}

The following is an easy consequence of the Kawamata-Viehweg vanishing theorem
(see, e.g., \cite[Prop. 2.1.2]{Iskovskikh-Prokhorov-1999}).

\begin{lemma}
Let $X$ be a variety with at worst 
\textup(log\textup) terminal singularities such that
$-K_X$ is nef and big. Then $\Pic(X)\simeq \Ho^2(X,\ZZ)$ is torsion free. 
Moreover, the numerical equivalence of Cartier divisors on $X$
coincides with the linear one.
\end{lemma}

\begin{corollary}
Let $X$ be a threefold with at worst Gorenstein terminal singularities 
such that
$-K_X$ is nef and big. 
Then the Weil divisor class group $\Cl(X)$ is torsion free. 
\end{corollary}

\begin{lemma}
\label{lemma-fixed-point}
Let $X$ be a threefold with at worst terminal singularities 
and let $G\subset \Aut(X)$ be a finite simple 
group.
If there is a $G$-fixed point $P$ on $X$, then
$G$ is isomorphic to a subgroup of $\Cr_2(\CC)$.
\end{lemma}
\begin{proof}
If $P\in X$ is Gorenstein, we consider the natural representation
of $G$ in the Zariski tangent space $T_{P,X}$.
First of all note that this representation is faithful.
Recall also that $P\in X$ is an isolated hypersurface singularity so the dimension of its
tangent space is at most $4$.
Therefore, $G\subset \GL (T_{P,X})$, where $\dim T_{P,X}=3$ or $4$.
Then by Theorems 
\ref{theorem-classification-3} and \ref{theorem-Blichfeldt}
the group $G$ is isomorphic to either $\A_5$, $\A_6$ or $\PSL_2(7)$.
In these cases $G$ admit an embedding into $\Cr_2(\CC)$
(see Theorem \ref{theorem-main-DI}).

Assume that $P\in X$ is non-Gorenstein of index $r>1$.
Take a small $G$-invariant neighborhood $P\ni U\subset X$
and consider the index-one cover $\pi\colon (U^\sharp,P^\sharp)\to 
(U,P)$ (see \S \ref{Terminal-singularities}). Here $(U,P)= (U^\sharp,P^\sharp)/\muu_r$, 
$(U^\sharp,P^\sharp)$ 
is a Gorenstein terminal 
point, and $U^\sharp\setminus \{P^\sharp\}\to 
U\setminus \{P\}$ is the topological universal cover. Let 
$\tilde G\subset \Aut (U^\sharp,P^\sharp)$ be the natural lifting of $G$.
There is the following exact sequence
\[
1\longrightarrow \muu_r \longrightarrow \tilde G \longrightarrow G \longrightarrow 1.
\]
Since $G$ is a simple group, the above sequence is a central extension.
If the representation of $G$ in $T_{P^\sharp,U^\sharp}$
has a non-trivial irreducible subrepresentation 
$T\subset T_{P^\sharp,U^\sharp}$, then we can apply Theorem 
\ref{theorem-classification-3} to the action on $T$.
Thus assume that 
the representation of $\tilde G$ in $T_{P^\sharp,U^\sharp}$
is irreducible. Then $\muu_r$ must act on $T_{P^\sharp,U^\sharp}$
by scalar multiplications. 
On the other hand, if  $P\in X$ is not a cyclic quotient singularity, then, according to the classification of 
terminal singularities \cite[Th. 6.1]{Reid-YPG1987}, 
the action of $\muu_r$ on $T_{P^\sharp,U^\sharp}$ in not free 
along a line. Hence,   $P\in X$ must be a cyclic quotient singularity.
In this case again according to \cite[Th. 5.2]{Reid-YPG1987}
$\muu_r$ acts on $T_{P^\sharp,U^\sharp}$
with weights $(w_1,w_2,w_3)$, where $(w_i, r)=1$ and $w_1+w_2\equiv 0\mod r$ (up to permutation of coordinates).
This is possible only if $r=2$ and 
$\dim T_{P^\sharp,U^\sharp}=3$. 
Then we can apply Theorem 
\ref{theorem-classification-3} again. 
\end{proof}

\begin{corollary}\label{corollary-orbit-7}
Let $X$ be a threefold with at worst Gorenstein terminal singularities 
such that
$-K_X$ is nef and big and let $G\subset \Aut(X)$ be a finite simple 
group which does not
admit an embedding into $\Cr_2(\CC)$. Then any $G$-orbit 
on $X$ contains at least $7$ elements.
\end{corollary}
\begin{proof}
Follows by Lemma \ref{lemma-fixed-point} and Theorem \ref{theorem-transitive-groups}.
\end{proof}

\begin{lemma}
\label{lemma-hyp-sect}
Let $X$ be a $G$-threefold with at worst terminal singularities
where $G$ is a finite simple group which does not
admit an embedding into $\Cr_2(\CC)$.
Assume that that $-K_X$ is nef and big.
Let $S$ be a $G$-invariant effective integral Weil 
$\QQ$-Cartier divisor numerically 
proportional to
$-K_X$. Then 
$K_X+S$ is nef. Furthermore, if $K_X+S\sim 0$, then 
the pair $(X,S)$ is LC \textup(log canonical, see e.g. \cite[ch. 2]{Utah}\textup) and the surface $S$
is reducible. If moreover $X$ is $G\QQ$-factorial, 
then the group $G$ transitively acts on 
components of $S$.
\end{lemma}

\begin{proof}
Assume that the divisor $-(K_X+S)$ is nef.
Clearly, $S$ is nef and big.
We apply quite standard
connectedness arguments of Shokurov \cite{Shokurov-1992-e-ba}: 

\begin{claim}[cf. {\cite[Prop. 2.6]{Mori-Prokhorov-2008d}}]
If either $-(K_X+S)$ is big or the pair $(X,S)$ is not LC,
then for a suitable $G$-invariant boundary $D$, the pair 
$(X,D)$ is LC, the divisor $-(K_X+D)$ is nef and big,
and the minimal locus $V$ of log canonical singularities of $(X,D)$  is non-empty and $G$-invariant. 
\end{claim}

\begin{proof}[Proof of the claim]
Take $c\in \QQ$ so that $(X,cS)$ is maximally LC.
Then $c\le 1$ and  $-(K_X+cS)$ is nef and big.
If the pair $(C,cS)$ is PLT (purely log terminal, see e.g. \cite[ch. 2]{Utah}), then 
we can take $D=cS$ and $V=\down{cS}$.
Thus we may assume that
there is a  center of log canonical singularities $W$ for $(X,cS)$ of dimension $\le 1$. 
Let $A'$ be an invariant very ample divisor on $X$.
Now take an element $F_1\in |-n(K_X+cS)-A'|$, $n\gg 0$.
We may assume that $F_1$ contains $W$.
Let $F_1,\dots, F_m$ be the $G$-orbit. Then $F:=\sum F_i$ is a $G$-invariant divisor 
contained in $|-nm(K_X+cS)-mA'|$. Thus there is a $G$-invariant decomposition
$-(K_X+cS)\equiv A+E$, where 
$A:= \frac 1n A'$ is ample, $E:=\frac{1}{nm} F$ is effective, and $W\subset S \cap \Supp(E)$.
Put $D_{\epsilon, \delta }:= (c-\epsilon)S+ \delta E$.
The divisor $-(K_X +D_{\epsilon, \delta })\equiv \epsilon S-(1-\delta)(K_X+cS)+\delta A$ is ample 
for all $0<\delta\le 1$, $\epsilon \ge 0$. 
Fix some $0< \delta \ll 1$ and then take
$\epsilon$ so that the pair 
$(X, D_{\epsilon, \delta })$ is maximally LC. 
Let
$V$ be a minimal center of  log canonical singularities for $(X, D_{\epsilon, \delta })$.
Take a general very ample divisor $H_1$ containing $V$.
Let $H_1,\dots, H_r$ be the $G$-orbit.
Fix some $0< \lambda \ll 1$ and then take
$\gamma$ so that the pair 
$(X, \lambda\sum H_i+(1-\gamma)D_{\epsilon, \delta })$ is maximally LC.
Put $D:=\lambda\sum H_i+(1-\gamma)D_{\epsilon, \delta }$.
It is easy to see that $-(K_X+D)$ is ample,
$V$ is a minimal center of  log canonical singularities for $(X, D)$, and 
$V$ does not meet other centers of log canonical singularities. Finally, 
by Shokurov's connectedness principle \cite{Shokurov-1992-e-ba}, \cite[ch. 17]{Utah} 
the whole locus of log canonical singularities
$(X, D)$ is connected. Hence it coincides with $V$. Thus $V$ is $G$-invariant.
\end{proof}

\par\medskip\noindent
\textit{Proof of Lemma \xref{lemma-hyp-sect} \textup(continued\textup).}
Assume 
either $-(K_X+S)$ is big or the pair $(X,S)$ is not LC.
By Lemma \ref{lemma-fixed-point} we may assume that $G$ has no fixed points.
Hence, in the above claim, $\dim V\ge 1$. Then $G\subset \Aut(V)=\Aut(\PP^1)$.
If $\dim V= 1$, then $V$ is a smooth rational curve 
\cite{Kawamata-1997-Adj}, so $G\subset \Aut(\PP^1)$, a contradiction.
Thus $V$ is an irreducible surface.
Then by the 
Inversion of Adjunction \cite{Shokurov-1992-e-ba}, \cite[Th. 17.6]{Utah} the surface 
$V$ is normal, has only log terminal singularities and $(K_X+D)|_V=K_V+D_V$, where $D_V$ 
is an effective Weil divisor on $V$ such that 
the pair $(V,D_V)$ is Kawamata log terminal (so-called  
\textit{different}, see \cite[\S 3]{Shokurov-1992-e-ba}, \cite[ch. 16]{Utah}). 
This implies that $(V,D_V)$ is a weak log del Pezzo surface, so $V$ is rational
(see e.g. \cite{Iskovskikh-Prokhorov-1999}). Therefore, $G\subset \Aut(V)\subset \Cr_2(\CC)$.
Again we get a contradiction.

Thus we may assume that the pair $(X,S)$ is LC and $K_X+S\sim 0$.
If the pair $(X,S)$ is PLT, then, as above, by the 
Inversion of Adjunction  the surface 
$S$ is normal and has only Du Val singularities.
Moreover, $K_S\sim 0$ and $\Ho^1(S,\OOO_S)=0$. 
Let $\tilde S\to S$ be the minimal resolution.
Then $\tilde S$ is a smooth K3 surface and $G$ naturally acts on $\tilde S$.
Recall that an automorphism $\varphi$ of a K3 surface $V$ is \textit{symplectic} if
$\varphi$ acts trivially on $H^0(V, K_V)\simeq \CC$. 
Since $G$ is a simple group, the action of $G$ on $\tilde S$  is symplectic.
According to \cite{Mukai1988} the group
$G$ is isomorphic to one of the following: 
$\A_5$, $\A_6$, $\PSL_2(7)$, so $G$ can be embedded to $\Cr_2(\CC)$.

Therefore, the pair $(X,S)$ is LC but not PLT.
Assume that $S$ is irreducible and let $\nu\colon S'\to S$
be the normalization. 
Recall that
     $G$ acts on $S$ faithfully by Lemma \ref{lemma-fixed-point}. 
If $S$ is rational, then we are in cases \eqref{eq-main-cr2}
because a faithfull action of a group on a rational surface gives
     an embedding of this group to the Cremona group of rank $2$.
So we assume that $S$ is not rational.
Write $0\sim \nu^*(K_X+S)|_S=K_{S'}+D'$,
where $D'$ is the different, an effective integral Weil divisor on $S'$ such that 
the pair $(S',D')$ is LC (see \cite[\S 3]{Shokurov-1992-e-ba}, \cite[ch. 16]{Utah},
\cite{Kawakita2007}). 
The group $G$ acts naturally on $S'$ and $\nu$ is $G$-equivariant.
Now consider the minimal resolution 
$\mu\colon \tilde S\to S'$ and let $\tilde D$ be the 
(uniquely defined) $\QQ$-divisor such that
\[
K_{\tilde S}+\tilde D=\mu^*(K_{S'}+D')\sim 0,\qquad \mu_*\tilde D=D'.
\]
Thus $\tilde D$ is usually called \textit{log crepant 
pull-back} of $D'$.
Here $\tilde D$ is again an effective reduced divisor.
Hence $\tilde S$ is a ruled non-rational surface.
Consider the Albanese map
$\alpha: \tilde S\to C$. Clearly $\alpha$ is 
$G$-equivariant and the action of $G$ on $C$ is not trivial
(otherwise $G$ non-trivially acts on a general fiber which is a rational curve).
The curve $C$ cannot be elliptic because otherwise $G$ is contained into $\Aut(C)$ which is a 
semi-direct product of the (abelian) group of translations
and a group of order $\le 6$. 
Hence, $g(C)>1$.
Let $\tilde D_1\subset \tilde D$ be a $\alpha$-horizontal
component. Since the surface is smooth,
by the genus formula $p_a(\tilde D_1)\le 1$. So, $\tilde D_1$ is either a
rational or elliptic curve.
This contradicts, $g(C)>1$.

Therefore the surface $S$ is reducible. 
If the action on components $S_i\subset S$ 
is not transitive and $X$ is $G\QQ$-factorial, 
we have an invariant divisor $S'<S$ which should be $\QQ$-Cartier.
This contradicts the above considered cases.
\end{proof}

\begin{corollary}
\label{corollary-fixed-components-linear-systems}
Let $X$ be a $G\QQ$-factorial $G$-threefold with at worst terminal singularities
where $G$ is a finite simple group which does not
admit an embedding into $\Cr_2(\CC)$.
Assume that $-K_X$ is nef and big.
Let $\HHH$
be a $G$-invariant linear system such that $\dim \HHH>0$ and 
$-(K_X+\HHH)$ is nef.
Then $\HHH$ has no fixed components.
\end{corollary}

\begin{proof}
Assume the converse $\HHH=F+\MMM$, where
$F$ is the fixed part and $\MMM$ is a linear system
without fixed components.
Then $F$ is an invariant divisor.
This contradicts Lemma \ref{lemma-hyp-sect}.
\end{proof}

\begin{lemma}
\label{lemma-Gorenstein-invariant-pencil}
Let $X$ be a $G\QQ$-factorial $G$-threefold with at worst terminal singularities
where $G$ is a finite simple group which does not
admit an embedding into $\Cr_2(\CC)$.
Assume that $-K_X$ is nef and big. Then
$\dim \Ho^0(X,-K_X)^G\le 1$.
\end{lemma}

\begin{proof}
Assume that 
there is a pencil $\HHH$ of invariant anticanonical sections.
By Corollary \ref{corollary-fixed-components-linear-systems}
$\HHH$ has no fixed components.
We claim that a general member of $\HHH$ is irreducible.
Indeed, otherwise $\HHH=m\LLL$, $m>1$ and the pencil $\LLL$
determines a $G$-equivariant 
rational map $X \dashrightarrow \PP^1$ so that
the action on $\PP^1$ is trivial. Hence, the fibers are 
$\QQ$-Cartier divisors and $-K_X\sim m\LLL$.
This contradicts Lemma \ref{lemma-hyp-sect} applied to $S\in \LLL$.
So, a general member $H\in \HHH$ is irreducible and $G$-invariant.
Again we get a contradiction by  Lemma \ref{lemma-hyp-sect}.
\end{proof}

\section{Case: $X$ is Gorenstein}
\label{section-Gorenstein}
\begin{mainassumption}
\label{main-assumption-2}
In this section 
$X$ denotes a threefold with at worst terminal Gorenstein singularities such that
the anticanonical divisor $-K_X$ is nef and big.
Let $G\subset \Aut(X)$ be a finite simple group 
which does not admit any embeddings into $\Cr_2(\CC)$.
Write $-K_X^3=2g-2$ for some $g$. This $g$ is called the \textit{genus} of a 
Fano threefold. By Kawamata-Viehweg vanishing and Riemann-Roch we have
$\dim |{-}K_X|=g+1$. In particular, $g$ is an integer.
\end{mainassumption}

\begin{lemma}
\label{lemma-base-point-free}
The linear system $|{-}K_X|$ is base point free.
\end{lemma}
\begin{proof}
Assume that $\Bs |{-}K_X|\neq \emptyset$.
If $\dim \Bs |{-}K_X|>0$, then by \cite{Shin1989}
$\Bs |{-}K_X|$ a smooth rational curve contained into the smooth locus of $X$.
By Lemma \ref{lemma-fixed-point} the action of $G$ on this curve is 
non-trivial. Hence $G\subset \Aut(\PP^1)$ and so $G\simeq \A_5$.
This contradicts Assumption \ref{main-assumption-2}.
Thus $\dim \Bs |{-}K_X|=0$. Again by \cite{Shin1989} $\Bs |{-}K_X|$ is 
a single point. 
This is impossible by  Lemma \ref{lemma-fixed-point}.
\end{proof}

\begin{lemma} 
\label{lemma-very-ample}
The linear system $|{-}K_X|$ 
determines a birational morphism $X\to \PP^{g+1}$ whose image 
is a Fano threefold $\bar X_{2g-2}\subset \PP^{g+1}$ 
with at worst canonical Gorenstein singularities.
In particular, $g\ge 3$.
\end{lemma}
\begin{proof}
Assume that the linear system 
$|{-}K_X|$ determines a morphism 
$\varphi\colon X\to \PP^{g+1}$ and $\varphi$ is not an embedding. 
Let $Y=\varphi(X)$.
Then $\varphi$ is a generically double cover and 
$Y\subset \PP^{g+1}$ is a subvariety of degree 
$g-1$ \cite{Iskovskikh-1980-Anticanonical}, \cite{Iskovskih1977a}, \cite{Przhiyalkovskij-Chel'tsov-Shramov-2005en}.
Note that the action
     of $G$ on $X$ induces a non-trivial (hence faithful) action of $G$ on
    $\varphi(X)$ since the map $\varphi: X \to \varphi(X)$ is given by $|{-}K_X|$.

If $Y$ is a projective cone,
then its vertex is either a point or  $\PP^1$.
Since $G$ is not embeddable to $\Cr_2(\CC)$, 
we get a contradiction by Corollary \ref{corollary-orbit-7}.

Thus we assume that $Y$ is not a cone.
According to the Enriques theorem the variety $Y\subset \PP^{g+1}$ is one of the following
(see, e.g., \cite[Lemma 2.8]{Iskovskih1977a}, \cite[Th. 3.11]{Iskovskikh-1980-Anticanonical}):
\begin{enumerate}
\item 
$\PP^{3}$;
\item 
a smooth quadric in~$\PP^{4}$;
\item 
a rational scroll
$\PP_{\PP^1}(\EEE)$, where 
$\EEE$ is a rank $3$ vector bundle on $\PP^1$.
\end{enumerate}
In the first case $\varphi :X\to \PP^3$ 
is a generically double cover with branch divisor $B\subset \PP^3$ of degree $6$.
By Theorem
\ref{theorem-Blichfeldt} $G\simeq \A_7$ or $\PSp_4(3)$.
However both these groups have no non-trivial representations of degree $4$,
a contradiction.

The second case does not occur by Lemma \ref{lemma-quadric4}.
In the last case $\rho(Y)=2$. Hence $G$ acts trivially on $\Pic(Y)$
and so the projection $Y\to \PP^1$
is $G$-equivariant. We get an embedding of $G$ into $\Aut(\PP^1)$ or 
$\Aut(F)$, where $F\simeq\PP^2$ is a fiber. 
\end{proof}

\begin{lemma}
\label{lemma-quadrics}
In notation of Lemma \xref{lemma-very-ample} one 
of the following holds:
\begin{enumerate}
\item
the variety
$\bar X=\bar X_{2g-2}\subset \PP^{g+1}$ is an intersection of quadrics
\textup(in particular, $g\ge 5$\textup);
\item
$g=3$, $\bar X=\bar X_{4}\subset \PP^{4}$ is quartic, and 
$G\simeq \PSp_4(3)$ \textup(see Example \xref{example-V6}\textup);
\item
$g=4$, $\bar X=\bar X_{6}\subset \PP^{5}$ is an intersection of 
a quadric and a cubic, and 
$G\simeq \A_7$ \textup(see Example \xref{example-V6}\textup).
\end{enumerate}
\end{lemma}

\begin{proof}
Assume that the linear system $|{-}K_X|$ determines a birational morphism
but its image 
$\bar X=\bar X_{2g-2}$ is not an intersection of quadrics.
Let $Y\subset \PP^{g+1}$ be the variety that cut out by
quadrics through $\bar X$. Then $Y$ is a four-dimensional 
irreducible subvariety in $\PP^{g+1}$ of minimal degree \cite{Iskovskikh-1980-Anticanonical},
\cite{Przhiyalkovskij-Chel'tsov-Shramov-2005en}.
As in the proof of Lemma \ref{lemma-very-ample} we can use the
Enriques theorem.
Assume that $Y$ is a cone with vertex $L$ over $S$.
Since $G$ is not contained in the 
list \eqref{eq-main-cr2}, $L$ is a point and $S$ is 
a three-dimensional variety of minimal degree
(and $S\not\simeq\PP^3$).
We get a contradiction as in the proof of Lemma \ref{lemma-very-ample}.
Hence $Y$ is smooth and we have the following possibilities:
\begin{enumerate}
\item 
$Y\simeq \PP^4$;
\item 
$Y\subset \PP^{5}$ is a smooth quadric;
\item 
a rational scroll
$\PP_{\PP^1}(\EEE)$, where 
$\EEE$ is a rank $4$ vector bundle on $\PP^1$.
\end{enumerate}
In the first case $g=3$ and $\bar X=\bar X_4\subset \PP^4$ is a quartic.
Consider the representation of $G$ in $\Ho^0(\bar X,-K_{\bar X})\simeq \CC^5$.
If this representation is reducible, then by our assumptions 
$\bar X$ has an invariant hyperplane section $S\in |{-}K_{\bar X}|$.
Since $\deg S=4$,
this $S$ must be irreducible (otherwise $S$ has a $G$-invariant rational component).
By Lemma \ref{lemma-hyp-sect} this is impossible.
Then by Theorem \ref{theorem-Brauer} and Assumption \ref{main-assumption-2}
we have the case (ii) of the lemma or the group $G$ is isomorphic to $\PSL_2(11)$. 
On the other hand, the group $\PSL_2(11)$ has no invariant quartics (see 
\cite[\S 29]{AdlerRamanan1996}), a contradiction.

In the second case $\bar X=\bar X_{6}\subset \PP^5$ 
is an intersection of a quadric and a cubic. 
By Lemma \ref{lemma-quadric4} we obtain either
$G\simeq \A_7$ or $\PSp_4 (3)$.
The second possibility is does not occur because 
the action of $\PSp_4 (3)$ on $\CC^6$ has no invariants of degree 3. 
(In fact, $\PSp_4 (3)$ can be embedded into
a group of order $51840$ generated by reflections, see \cite[Table VII, No. 35]{Shephard-Todd}).
Thus $G\simeq \A_7$. We get a situation of Example \ref{example-V6}
because the group $\A_7$ has only one irreducible representation of degree $6$.

In the last case, as in Lemma \ref{lemma-very-ample},
we have a $G$-equivariant contraction $Y\to \PP^1$ 
whose fibers are isomorphic to $\PP^3$.
The restriction map $X\to \PP^1$ is a fibration whose 
general fiber $F$ is a surface with big and nef anticanonical 
divisor. Such a surface must be rational. Hence either 
$G\subset \Aut(\PP^1)$ or $G\subset \Aut(F)$.
\end{proof}

\begin{corollary}
\label{corollary-quadrics}
In case \textup{(i)} of Lemma \xref{lemma-quadrics}
the variety $\bar X=\bar X_{2g-2}\subset \PP^{g+1}$ is an intersection
of $(g-2)(g-3)/2$ quadrics.
\end{corollary}

\begin{proof}
Let $S\subset \PP^g$ be a general hyperplane section of $\bar X$ and let $C\subset \PP^{g-1}$ 
be a general hyperplane section of $S$.
Then $S$ is a smooth K3 surface and $C$ is a canonical curve of genus $g$.
Let $\mathscr I_{\bar X}$ (resp. $\mathscr I_S$, $\mathscr I_C$)
be the ideal sheaf of $\bar X\subset \PP^{g+1}$ (resp. $S\subset \PP^{g}$, $C\subset \PP^{g-1}$).
The space $\Ho^0(\III_{\bar X}(2))$ is the space of quadrics in 
$\Ho^0(\bar X,-K_{\bar X})$
passing through $\bar X$. 
The standard cohomological arguments 
(see, e.g.,  \cite{Iskovskih1977a}, \cite[Lemma 3.4]{Iskovskikh-1980-Anticanonical})
show that $\Ho^0(\III_{\bar X}(2))\simeq \Ho^0(\III_S(2))\simeq\Ho^0(\III_C(2))$.
This gives us 
\[
\dim \Ho^0(\bar X, \III_{\bar X}(2))= \frac12(g-2)(g-3).
\]
\end{proof}

\begin{theorem}[{\cite{Namikawa-1997}}]
\label{theorem-Namikawa}
Let $X$ be a Fano threefold with terminal Gorenstein
singularities. Then $X$ is smoothable, that is,
there is a flat family $X_t$ 
such that $X_0\simeq X$ and a general member $X_t$ is 
a smooth Fano threefold.
Further, the number of singular points is bounded as follows:
\begin{equation}
\label{eq-number-singular-points}
|\Sing(X)|\le 21-\frac12 \Eu(X_t)=20-\rho(X_t)+\h^{1,2}(X_t).
\end{equation}
where $\Eu(X)$ is the topological Euler number and $\h^{1,2}(X)$ is the Hodge number. 
\end{theorem}

\begin{remark}
\label{remark-Namikawa}
\begin{enumerate}
 \item 
In the above notation the total family $\mathfrak X$ has at worst isolated terminal 
factorial singularities and there are natural identifications 
$\Pic(X)\simeq \Pic(\mathfrak X)\simeq \Pic(X_t)$ (see {\cite[\S 1]{jahnke-Radloff-2006arx}}). 
In particular, 
$\rho(X_t)=\rho(X)$,
 $-K_{X_t}^3=-K_X^3$, and varieties $X$ and $X_t$ have the same Fano index. 

\item
The estimate \eqref{eq-number-singular-points} is very far from being sharp.
For example, for cubic hypersurface $X\subset \PP^4$ \eqref{eq-number-singular-points}
gives us $|\Sing(X)|\le 24$ but the sharp bound is $|\Sing(X)|\le 10$
and achevied for the Segre cubic.
However, for our purposes,  \eqref{eq-number-singular-points} is sufficient.
\end{enumerate}
\end{remark}

\begin{theorem}[see, e.g., \cite{Iskovskikh-1980-Anticanonical}, \cite{Iskovskikh-Prokhorov-1999}]
\label{theorem-calssification-Fano}
Let $X$ be a smooth Fano threefold with $\Pic (X) =\ZZ\cdot (-K_X)$. 
Then the possible values of its genus $g$ and Hodge numbers $\h^{1,2}(X)$ 
are given by the following table:{\rm
\begin{center}
\begin{tabular}{|c|cccccccccc|}
\hline
&&&&&&&&&&
\\[-5pt]
$g$&2&3&4&5&6&7&8&9&10&12
\\[5pt]
\hline
&&&&&&&&&&
\\[-5pt]
$\h^{1,2}(X)$&52&30&20&14&10&7&5&3&2&0
\\[5pt]
\hline
\end{tabular}
\end{center}}
\end{theorem}

\begin{mainassumption}
\label{main-assumption-3}
From now on and till the end of this section 
additionally to \ref{main-assumption-2}
we assume that $-K_X$ is ample, $X$ is $G\QQ$-factorial, and 
$\rho(X)^G=1$, i.e., 
$X$ is a Gorenstein $G\QQ$-Fano threefold.
Moreover, the anticanonical linear system determines an embedding 
$X=X_{2g-2}\subset \PP^{g+1}$ and its image is an intersection
of $(g-2)(g-3)/2$ quadrics.
\end{mainassumption}

\begin{lemma}
\label{lemma-rho=1}
Under the assumptions of \xref{main-assumption-3} we have
$\rho(X)=1$.
\end{lemma}
\begin{proof}
Assume that $\rho(X)>1$.
We have a natural action of $G$ on $\Pic (X)\simeq\ZZ^\rho$
such that $\Pic (X)^G\simeq\ZZ$.
In particular, there is a non-trivial representation 
$V\subsetneq \Pic (X)\otimes\RR$.
Hence $G$ admits an embedding into $\operatorname{PSO}_{\rho-1}(\RR)$.
By Lemma \ref{lemma-quadric4} we have $\rho(X)\ge 7$. 
Consider a smoothing $X_t$ of $X$. Here $X_t$ is a 
smooth Fano threefold with $\rho(X_t)=\rho(X)$
and $-K_{X_t}^3=-K_X^3$ (see Remark \ref{remark-Namikawa}, (i)). From 
the classification of smooth Fano threefolds with 
$\rho>1$ \cite{Mori1981-82}
one can see that $X_t\simeq S\times \PP^1$, where $S$ is a del Pezzo surface. 

Again by Remark \ref{remark-Namikawa}, (i)
there is natural identification 
$\Pic(X)\simeq  \Pic(X_t)$ that preserves the intersection form.
So we assume that $G$ acts on $\Pic(X_t)$ (but not on $X_t$).
Let $F$ be a fiber of the projection $X_t=S\times \PP^1\to \PP^1$.
Take an element $\tau\in G$ sending $F$ to $F'$ that is not proportional to $F$.
Then $F'\sim \alpha F+ f^* L$ for some $0\neq L\in \Pic(S)$
and $\alpha \in \ZZ$.
Since $F^2\equiv 0$, we have 
\[
0= F'^2\cdot F= f^*L^2\cdot F.
\]
Hence, $L^2= 0$ and $2 \alpha F\cdot f^* L\equiv F'^2 \equiv 0$. 
So, $\alpha=0$ and $F'= f^* L$. Further, by Riemann-Roch $K_{S}\cdot L$
is even and
\[
K_{S}^2=K_{X_t}^2\cdot F=K_{X_t}^2\cdot F'=K_{X_t}^2\cdot f^* L=
(2F-f^*K_{S})^2\cdot f^* L=-4 K_{S}\cdot L.
\]
Therefore, $K_{S}^2=8$ and $\rho(X)=\rho(X_t)=3$, a contradiction. 
\end{proof}

Recall that the Fano index of a Gorenstein Fano variety $X$ is the 
maximal positive integer dividing the class of $-K_X$ in $\Pic(X)$.

\begin{lemma}
\label{lemma-index-Fano}
Under the assumptions of \xref{main-assumption-3} we have either
\begin{enumerate}
\item
the Fano index of $X$ is one, or
\item
$G\simeq \PSL_2(11)$ and $X^{\mathrm k}_3\subset \PP^4$ is the Klein cubic
\textup(see Example \xref{example-Klein-cubic}\textup).
\end{enumerate}
\end{lemma}
\begin{proof}
Let $q$ be the Fano index of $X$.
Write $-K_X=qH$, where $H$ is an ample Cartier divisor.
Clearly, the class of $H$ is $G$-stable.
Assume that $q>1$. If $q>2$, then $X$ is either $\PP^3$ or 
a quadric in $\PP^4$. Thus we may assume that $q=2$.
Below we use some facts on Gorenstein Fano 
threefolds of Fano index $2$ with at worst canonical singularities,
see \cite{Iskovskikh-1980-Anticanonical}, \cite{Shin1989}.
Denote $d=H^3$. 

As in the proof of Lemma \ref{lemma-rho=1}
there is a flat family $X_t$
such that $X_0\simeq X$ and a general member $X_t$ is a 
smooth Fano threefold with the same Picard number, anticanonical degree, and Fano index.
Since $\rho(X)=1$, by the classification of smooth Fano threefolds
\cite{Iskovskikh-1980-Anticanonical}, \cite{Iskovskikh-Prokhorov-1999}
$d\le 5$.

If $d=1$, then $\Bs |H|$ is a single point
contained into the smooth part of $X$. 
This point must be $G$-invariant. This contradicts 
Lemma \ref{lemma-fixed-point}.
If $d=2$, then the linear system
$|H|$ determines a $G$-equivariant double cover $X\to \PP^3$
with branch divisor $B=B_4\subset \PP^3$ of degree $4$.
Clearly, $B$ has only isolated singularities.
If $B$ has at worst Du Val singularities, then 
according to \cite{Mukai1988} the group
$G$ is isomorphic to one of the following: 
$\A_5$, $\A_6$, $\PSL_2(7)$, so $G$ can be embedded to $\Cr_2(\CC)$, a contradiction.
Hence $B$ is not Du Val.
The non-Du Val locus of $B$ coincides with the locus of 
log canonical singularities $\operatorname{LCS}(\PP^3,B)$ of the pair $(\PP^3,B)$.
By a generalization of Shokurov's connectedness principle \cite[Th. 6.9]{Shokurov-1992-e-ba}
the set $\operatorname{LCS}(\PP^3,B)$ is either connected
or has two connected components. Then $G$ has a fixed point on $B$ and on 
$X$. This contradicts Lemma \ref{lemma-fixed-point}.

For $d>2$, the linear system $|H|$ is very ample
and determines a $G$-equivariant embedding $X \hookrightarrow \PP^{d+1}$.
Therefore, $G\subset \PGL_{d+2}(\CC)$.
Take a lifting $\tilde G\subset \GL_{d+2}(\CC)$ so that
$\tilde G/Z(\tilde G)\simeq G$ and
$Z(\tilde G)\subset [\tilde G,\tilde G]$.
We have a natural non-trivial representation of $\tilde G$ in $\Ho^0(X,H)$, where 
$\dim \Ho^0(X,H)=d+2\le 7$.
We claim that this representation is irreducible.
Indeed, assume that $\Ho^0(X,H)$ is reducible as a $\tilde G$-module.
By Lemma \ref{lemma-hyp-sect} the variety $X$ has no invariant hyperplane sections,
i.e.,  the representation of $\tilde G$ in $\Ho^0(X,H)$ has no one-dimensional
subrepresentations. Hence $\Ho^0(X,H)$
has an irreducible subrepresentation $V$ of dimension $2$ or $3$.
In this case, $G\simeq \tilde G/Z(\tilde G)$ acts faithfully on  
$\PP(V)\subset \PP(\Ho^0(X,H))$.
So, $G$ admits an embedding to $\Cr_2(\CC)$.
This contradicts our assumption \ref{main-assumption-2}
and proves the claim.

Consider the case $d=3$. 
Assuming that 
$G$ is not contained in $\Cr_2(\CC)$ by
Theorem \ref{theorem-Brauer} we have either $G\simeq\PSL_2(11)$ or $G\simeq\PSp_4(3)$.
In the first case, 
the only cubic invariant of this group is the Klein cubic \eqref{equation-Klein-cubic},
see \cite[\S 29]{AdlerRamanan1996}.
We get Example \ref{example-Klein-cubic}. The second case is impossible because
the group $\PSp_4(3)$ has no invariants of degree $3$, see \cite{Shephard-Todd}.

Consider the case $d=4$. 
Then $X=X_4\subset \PP^5$ is an intersection of two quadrics, say $Q_1$ and $Q_2$.
The action of $G$ on the pencil generated by $Q_1$, $Q_2$ must be trivial.
Hence $G$ acts on a degenerate quadric $Q'\in \langle Q_1, Q_2\rangle$.
In particular, $G$ acts on the singular locus of $Q'$
which is a linear subspace, a contradiction.

Consider the case $d=5$. Then $X\subset \PP^6$ is an intersection of 
$5$ quadrics \cite{Shin1989}. Let $V=\Ho^0(X,\III_X(2))$, where $\III_X$ 
be the ideal sheaf of $X$ in $\PP^6$. 
Then $V$ is a $5$-dimensional $G$-invariant subspace of $\Ho^0(X,\OOO_X(2))=\Ho^0(X,-K_X)$.
If the action of $G$ on $V$ is trivial, then, as above, there is a $G$-stable
singular quadric $Q\subset \PP^6$. But then the singular locus of $Q$
is a $G$-stable linear subspace in $\PP^6$, a contradiction.
Thus $G\subset \SL_5(\CC)$. Assuming that 
$G$ is not contained in the list \eqref{eq-main-cr2} by
Theorems \ref{theorem-Blichfeldt} and \ref{theorem-Brauer} 
the group $G$ is isomorphic to either $\PSp_4(3)$ or  $\PSL_2(11)$.
In both cases, the Schur multiplier of $G$ is a group of order $2$
and the covering group $\tilde G$ is isomorphic to 
$\Sp_4(3)$ and $\SL_2(11)$, respectively, see \cite{atlas}. 
Since the order of $\Sp_4(3)$ and $\SL_2(11)$
is not divisible by $7$, these groups have no irreducible representations
of degree $7$, a contradiction. 
\end{proof}

\begin{mainassumption}
\label{main-assumption-4}
Thus in what follows additionally to
\ref{main-assumption-2} and
\ref{main-assumption-3} we assume that 
the Fano index of $X$ is one.
\end{mainassumption}

\begin{lemma}
\label{lemma-Gorenstein-invariant-hyperplane}
Under the assumptions of \xref{main-assumption-4} we have
$\Ho^0(X,-K_X)^G=0$.
\end{lemma}
\begin{proof}
Assume that $G$ has an invariant hyperplane section $S$.
By Lemma \ref{lemma-hyp-sect} the pair $(X,S)$ is LC, $S=\sum S_i$ and $G$ acts 
transitively on $\Omega:=\{S_i\}$. Let $m:= |\Omega|$.
Recall that $4\le g\le 12$ and $g\neq 11$ by Theorem \ref{theorem-calssification-Fano} and Lemma \ref{lemma-rho=1}.
We have $m\deg S_i=2g-2\le 22$. Since $m\ge 7$, $\deg S_i\le 3$.
The action of $G$ on $\Omega$ 
induces a transitive embedding $G\subset \Sy_m$.

If $\deg S_i=2$, then $m=g-1\le 11$, $m\neq 10$. 
Recall that the natural representation of $G$
in $\Ho^0(X,-K_X)=\CC^{g+2}=\CC^{m+3}$ has no two-dimensional trivial subrepresentations.
Taking this into account and using table in Theorem \ref{theorem-transitive-groups}
we get only one case: $m=7$, $g=8$, $G\simeq \A_7$, and the action 
of $\A_7$ on $\{S_1,\dots, S_7\}$ is the standard one.
Moreover, $S_i$ is either $\PP^1\times \PP^1$ or 
a quadratic cone $\PP(1,1,2)$. Therefore the stabilizer 
$G_{S_i}\simeq \A_6$ acts 
trivially on $S_i$. The ample divisor $\sum S_i$ is connected.
Hence, $S_i\cap S_j\neq \emptyset$ for some $i\neq j$.
Then the stabilizer $G_P$ of the point $P\in S_i\cap S_j$
contains the subgroup generated by $G_{S_i}$ and $G_{S_j}$.
So, $G_P=G$. This contradicts Lemma \ref{lemma-fixed-point}.

Hence $\deg S_i\neq 2$. Then $\deg S_i$ is odd, 
$m$ is even, and $m\ge 8$. 
This implies that $\deg S_i=1$, i.e., $S_i$ is a plane. 
Moreover, $m=2g-2\le 22$, $m\neq 20$.
As above, using the fact that the representation of $G$
in $\Ho^0(X,-K_X)=\CC^{m/2+3}$ has no two-dimensional trivial subrepresentations
and Theorem \ref{theorem-transitive-groups} we get only one case:
$m=8$, $g=5$, and $G\simeq \A_8$.
Similar to the previous case we derive a contradiction.
The lemma is proved.
\end{proof}

\begin{corollary}
If in the assumptions of \xref{main-assumption-4}\ $g\le 7$, 
then the representation of $G$ in $\Ho^0(X,-K_X)$ is irreducible.
\end{corollary}
\begin{proof}
Follows from Theorem \ref{theorem-Blichfeldt} and Lemma \ref{lemma-Gorenstein-invariant-hyperplane}.
\end{proof}

\begin{lemma}
Under the assumptions of \xref{main-assumption-4} we have
$g\ge 7$.
\end{lemma}
\begin{proof}

Assume that $g=5$. 
Then by Corollary \ref{corollary-quadrics}
we have $\dim \Ho^0(\III_X(2))=3$ and $X\subset \PP^6$ is a complete
intersection of three quadrics. The group $G$ acts on 
$\Ho^0(\III_X(2))\simeq \CC^3$ and we may assume that this
action is trivial (otherwise $G$ acts on $\PP^2=\PP(\Ho^0(\III_X(2)))$). 
Thus we have a net 
of invariant quadrics $\lambda_1 Q_1+\lambda_2 Q_2+\lambda_3 Q_3$.
In particular, there is an invariant degenerate quadric 
$Q'\in \lambda_1 Q_1+\lambda_2 Q_2+\lambda_3 Q_3$.
By Lemma \ref{lemma-quadric4} $Q'$ is a cone with zero-dimensional vertex
$P$. Thus $P\in \PP^7$ is an invariant point and 
there is an invariant hyperplane section, a contradiction. 

Now assume that $g=6$. 
Again by Corollary \ref{corollary-quadrics}
we have 
$\dim \Ho^0(\III_X(2))=6$.
If the action of $G$ on $\dim \Ho^0(X,\III_X(2))^G>1$,
then $G$ acts on a singular irreducible $6$-dimensional 
quadric $Q\subset \PP^7$. In particular, 
the singular locus of $Q$,
a projective space $L$ of dimension $\le 4$ must be $G$-invariant.
This contradicts the irreducibility of $\Ho^0(X,-K_X)$.
Therefore, $\dim \Ho^0(X,\III_X(2))^G\le 1$.
In particular, $G$ acts on $\Ho^0(X,\III_X(2))\simeq \CC^6$ non-trivially
and so $G$ has an irreducible representation of degree $5$ or $6$.
Since $G$ is simple and because we assume that $G$ is not contained in the list 
\eqref{eq-main-cr2} by the classification 
theorems \ref{theorem-Brauer} and \ref{theorem-Lindsey}
we have only four possibilities: $G\simeq \A_7$, $\PSp_4(3)$, $\PSL_2(11)$, or $\SU_3(3)$.
But in all cases $G$ has no irreducible representations of degree $8$ (see \cite{atlas}),
a contradiction.
\end{proof}

\begin{lemma}
\label{lemma-Gor-tr}
Under the assumptions of \xref{main-assumption-4} the variety 
$X$ is smooth. 
\end{lemma}
\begin{proof}
 Assume that $X$ is singular. 
Let $\Omega\subset \Sing(X)$ be a $G$-orbit and let 
$n:=|\Omega|$. Let $x_1,\dots,x_{n}\in \Ho^0(X,-K_X)^*$ be the vectors
corresponding to the points of $\Omega$.
By \eqref{eq-number-singular-points}
we have $n\le 26$.
Let $P\in \Sing (X)$ and let $G_P$ be the stabilizer of $P$.
Then the natural representation of $G_P$ in $T_{P,X}$ is faithful.
On the other hand, by Corollary \ref{corollary-transitive-groups} the group
$G_P$ has a faithful representation of degree $\le 4$ 
only in the following cases:
\begin{enumerate}
 \item 
$G\simeq\PSL_2(11)$,\quad $|\Omega|=11$,\quad $G_P\simeq \A_5$;
\item 
$G\simeq\A_7$,\quad $|\Omega|=21$,\quad $G_P\simeq \Sy_5$;
 \item
$G\simeq\A_7$,\quad $|\Omega|=15$,\quad $G_P\simeq \PSL_2(7)$.
\end{enumerate}
Locally near $P$ the singularity $X\ni P$ 
is given by a $G_P$-semi-invariant equation $\phi(x,\dots,t)=0$.
Write $\phi=\phi_2+\phi_3+\dots$, where $\phi_d$ is the homogeneous 
part of degree $d$.
By the classification of terminal singularities, $\phi_2\neq 0$.
The last case $G_P\simeq \PSL_2(7)$ is impossible because, 
then the representation of $G_P$ in $T_{P,X}$ is
reducible: $T_{P,X}=T_1\oplus T_3$, where $T_3$ is an irreducible
representation of degree $3$. Since the action of $G_P$ 
on $T_3$ has no 
invariants of degree $2$ and $3$ (see \cite{Shephard-Todd}),
we have $\phi_2=\ell^2$ and $\phi_3=\ell^3$,
where $\ell$ is a linear form. But this contradicts the classification of terminal singularities 
\cite[Th. 6.1]{Reid-YPG1987}. Therefore, $G_P\simeq \A_5$ or $\Sy_5$ and we are in cases (i) or (ii).

\begin{claim}
If $X$ is singular, then $g=8$.
\end{claim}

\begin{proof}
The natural representation of $G$ in $\Ho^0(X,-K_X)\simeq \CC^{g+2}$ has no trivial 
subrepresentations. Recall that $g=7$, $8$, $9$, $10$, or $12$.

Consider the case $G\simeq \A_7$. Then the degrees of irreducible representations in the 
interval $[2,14]$ are $6$, $10$, $14$ (see Theorem \ref{theorem-transitive-groups}).
Hence, $g=8$, $10$, or $12$. On the other hand, 
$X$ has at least $21$ singular points (because we are in the case (ii) above). 
By \eqref{eq-number-singular-points} we have $\h^{1,2}(X')\ge 2$.
So, $g\neq 12$. 
Let $\chi$ be the character of $G$ on $\Ho^0(X,-K_X)^*$.
We need the character table for $G=\A_7$
(see, e.g., \cite{atlas}):
\begin{subequation}
\label{sub-equation-characters-A7}
\begin{array}{c|rrrrrrrrr}
G & \mathcal C_1& \mathcal C_2& \mathcal C_3'& \mathcal C_{6}
& \mathcal C_{3}''& \mathcal C_4& \mathcal C_5& \mathcal C_7'& \mathcal C_7''
\\
\hline
\\
\chi_1 & 1 & 1& 1& 1& 1& 1& 1& 1& 1
\\
\chi_2& 6& 2& 3 &-1& 0& 0& 1& -1 &-1
\\
\chi_3& 10 &-2& 1& 1& 1& 0& 0& \alpha& \bar \alpha
\\
\chi_4& 10& -2& 1& 1& 1& 0& 0& \bar \alpha& \alpha
\\
\dots&\hdotsfor{9}
\end{array}
\end{subequation}
\par\noindent
Here $\alpha=(-1+\sqrt{-7})/2$.
(We omit characters of degree $\ge 14$).
Assume that $g=10$.
Since the representation of $G$ in $\Ho^0(X,-K_X)$ has no trivial 
subrepresentations, the only possibility is  
$\Ho^0(X,-K_X)^*=W\oplus W'$, as $G$-module,  where $W\simeq W'$
is a $6$-dimensional representation
(i.e. $\chi=\chi_2\oplus \chi_2$).
Thus $\Ho^0(X,-K_X)^*$ contains a one-dimensional family $W_{\lambda}$ 
of subrepresentations isomorphic to $W$.
Replacing $W$ with $W_{\lambda}$ 
we can take the decomposition above so that the first copy $W$ contains the vector
$x_1$ (corresponding to $P_1\in \Omega$). Then $P_1\in \PP^5=\PP(W)$
and obviously $\Omega\subset \PP(W)$.
Consider the set $S:=\PP(W)\cap X$, the base locus of the linear system of 
hyperplane sectons passing through $\PP(W)$.
By Corollary
\ref{corollary-fixed-components-linear-systems} $\dim S\le 1$.
Assume that $\dim S=0$. Take a general hyperplane section $H$ passing through
$\PP(W)$. By Bertini's theorem $H$ is a normal surface with
isolated singularities. Moreover, $H$ is singular at points of $\Omega$,
so $|\Sing (H)|\ge |\Omega|=21$.
By the adjunction $K_H\sim0$. Hence, by
a generalization of Shokurov's connectedness principle
 \cite[Th. 6.9]{Shokurov-1992-e-ba},
$H$ has at most two non-Du Val singularities.
Since $G$ has no fixed points on $X$, the surface $H$ has only Du Val singularities.
Therefore, the minimal resolution $\tilde H$ of $H$ is a K3 surface
and so $\rho(\tilde H)\le \dim \Ho^{1,1}(\tilde H)=20$.
On the other hand, $\rho(\tilde H)>|\Sing (H)|\ge 21$, a contradiction. 
Thus $\dim S=1$. Let $S'$ be the union of an orbit of a one-dimensional component. 
Since the representation $W$ is irreducible, $S'$ spans $\PP(W)$.
By Lemma \ref{lemma-quadrics} $S\subset \PP^5$ is an intersection of quadrics.
Since $S'\subset S$ and $\dim S=1$, $\deg S'\le 16$. If $S'$ is reducible, then $G$ 
interchanges its components $S_i$. In this case, $\deg S_i\le 2$.
By Theorem \ref{theorem-transitive-groups} the number of components 
is either $7$ or $15$. The stabilizer $G_{S_i}$ ($\simeq \A_6$ or $\PSL_2(7)$)
acts on $S_i$ which is a rational curve, a contradiction.
Therefore, $S'\subset \PP^5$ is an irreducible curve contained in $S$,
an intersection of quadrics. Let $S''\to S'$ be its normalization.
By the Castelnuovo bound $g(S'')\le 21$ (see e.g. \cite[ch. 3, \S 2]{Arbarello1985}).
On the other hand, by the Hurwitz bound
\cite[ch. 1, \S 6, F]{Arbarello1985} we have $|G|\le \Aut(S'')\le 84(g(S'')-1)$,
a contradiction.

Now consider the case $G\simeq \PSL_2(11)$.
As above, since the natural representation of $G$ in 
$\Ho^0(X,-K_X)\simeq \CC^{g+2}$ has no trivial 
subrepresentations, we have $g=8$, $9$, or $10$ (see
Theorem \ref{theorem-transitive-groups}).
Moreover, if $g=10$, then the representation of $G$ in $\Ho^0(X,-K_X)$ is irreducible.
On the other hand, 11 points of the set $\Omega\subset \PP(\Ho^0(X,-K_X)^*)=\PP^{11}$
generate and invariant subspace, a contradiction.
If $g=9$, then 
11 points of the set $\Omega\subset \PP(\Ho^0(X,-K_X)^*)$ are in general position.
Then the corresponding vectors $x_i\in \Ho^0(X,-K_X)^*$ are linearly independent and 
the representation of $G$ in $\Ho^0(C,-K_X)$ is induced from 
the trivial representation of $G_P$ in $\langle x_1\rangle$.
But in this case the $G$-invariant vector $\sum_{\delta\in G} \delta (x_1)$ is not zero, a contradiction. 
Thus $g=8$. 
\end{proof}

\begin{claim}
If $X$ is singular, then 
$G\not\simeq \A_7$. 
\end{claim}
\begin{proof}
Assume that $G\simeq \A_7$. Then $G_P\simeq \Sy_5$.
We compare the character tables for $\A_7$
(see \eqref{sub-equation-characters-A7}) and for $\Sy_5$:
\begin{subequation}
\label{sub-equation-characters-S5}
\begin{array}{c|rrrrrrr}
G_P & \mathcal C_1& \mathcal C_2'& \mathcal C_2''& \mathcal C_{3}
& \mathcal C_{6}& \mathcal C_4& \mathcal C_5
\\
\hline
\\
\chi_1'& 1& -1& 1& 1& -1 &-1& 1
\\
\chi_2'& 4& -2& 0& 1& 1& 0& -1
\\
\chi_3'& 5& -1& 1 &-1& -1& 1& 0
\\
\chi_4'& 6& 0& -2& 0& 0& 0& 1
\\
\chi_5'& 5& 1& 1& -1& 1 &-1& 0
\\
\chi_6'& 4& 2& 0& 1 &-1& 0& -1
\\
\chi_7'& 1& 1& 1& 1& 1& 1& 1
\end{array} 
\end{subequation}
Let $\chi$ be the character of the representation of $G$ in $\Ho^0(X,-K_X)^*$.
By Lemma \ref{lemma-Gorenstein-invariant-hyperplane}
and \eqref{sub-equation-characters-A7} $\chi$ is irreducible and either
$\chi=\chi_3$ or $\chi=\chi_4$ (recall that $\chi_3$ and $\chi_4$ 
are characters of $\A_7$). 
Using \eqref{sub-equation-characters-A7}, 
in notations of \eqref{sub-equation-characters-S5},
for the restriction $\chi|_{\Sy_5}=\chi_3|_{\Sy_5}=\chi_4|_{\Sy_5}$ we obtain
\[
\chi|_{\Sy_5}(\mathcal C_1, \mathcal C_2', \mathcal C_2'', 
\mathcal C_{3}, \mathcal C_{6}, \mathcal C_4, \mathcal C_5) =
(10,-2,-2,1,1,0,0).
\]
Hence, $\chi|_{\Sy_5}=\chi_2'\oplus \chi_4'$.
In particular, the representation of $G_P\simeq \Sy_5$ 
in $\Ho^0(X,-K_X)^*$ has no trivial subrepresentations, a contradiction.
\end{proof}
Thus we may assume that $G\simeq \PSL_2(11)$ and $G_P\simeq \A_5$.

\begin{claim}\label{claim-PSL-11}
If $X$ is singular, then 
the natural representation of $G_P$ in $T_{P,X}$ is irreducible
and $P\in X$ is an ordinary double point, that is, 
$\rk \phi_2=4$. 
\end{claim}
\begin{proof}
Let $x\in \Ho^0(X,-K_X)^*$ be a vector corresponding to $P$.
There is a $G_P$-equivariant embedding $T_{P,X} \hookrightarrow \Ho^0(X,-K_X)^*$
so that $x\notin T_{P,X}$.
Thus $\Ho^0(X,-K_X)^*$ has a trivial $G_P$-representation $\langle x\rangle$
which is not contained in $T_{P,X}$.
Let $\chi$ be the character of $G$ on $\Ho^0(X,-K_X)^*$.
We need character tables for $G=\PSL_2(11)$ and $G_P=\A_5$
(see, e.g., \cite{atlas}):
\[
\begin{array}{c|rrrrrrrr}
G & \mathcal C_1& \mathcal C_5'& \mathcal C_5''& \mathcal C_{11}'
& \mathcal C_{11}''& \mathcal C_2& \mathcal C_3& \mathcal C_6
\\
\hline
\\
\chi_1& 1& 1& 1& 1& 1& 1& 1& 1\\
\chi_2& 5& 0& 0& \beta& \bar\beta& 1& -1& 1\\
\chi_3& 5& 0& 0& \bar\beta& \beta& 1& -1& 1\\
\chi_4& 10& 0& 0& -1& -1 &-2& 1& 1\\
\chi_5& 10& 0& 0& -1& -1& 2& 1& -1\\
\dots&\hdotsfor{8}
\end{array}
\quad
\begin{array}{c|rrrrr}
G_P& \mathcal C_1& \mathcal C_2& \mathcal C_3& \mathcal C_5'& \mathcal C_5''
\\
\hline
\\
\chi_1'& 1& 1& 1& 1& 1\\
\chi_2'& 3& -1& 0& \alpha& \alpha^*\\
\chi_3'& 3 &-1& 0& \alpha^*& \alpha\\
\chi_4'& 4& 0& 1& -1& -1\\
\chi_5'& 5& 1& -1& 0& 0
\\
\phantom{\hdots}
\end{array}
\]
Here 
$\beta = (-1+\sqrt{-11})/2$,
$\alpha = (1-\sqrt{5})/2$, and $\alpha^* = (1+\sqrt{5})/2$.
(We omit characters of degree $>10$).
Assume that the representation of $G_P$ in $T_{P,X}$
is reducible.
Then the restriction $\chi|_{G_P}$ contains $\chi_1'$ with multiplicity $\ge 2$
and either $\chi_2'$ or $\chi_3'$.
Comparing the above tables we see that 
the restrictions $\chi_2|_{G_P}$ and $\chi_3|_{G_P}$ are irreducible
(and coincide with $\chi_5'$). Hence, $\chi=\chi_4$ or $\chi_5$.
In particular, $\chi(\mathcal C_5')=\chi(\mathcal C_5'')=0$
and $\chi|_{G_P}$ contains both $\chi_2'$ and $\chi_3'$.
Thus $\chi|_{G_P}=\chi_2'+\chi_3'+4\chi_1'$ and so
$\chi(\mathcal C_3)=4$. This contradicts $\chi_4(\mathcal C_3)=\chi_4(\mathcal C_5)=1$.
Therefore the character of the representation of $G_P$ in $T_{P,X}$
coincides with $\chi_4'$ (and irreducible). 

Then the vertex of the 
tangent cone $TC_{P,X}\subset T_{P,X}$ to $X$ at $P$ must be zero-dimensional.
Hence, $TC_{P,X}$ a cone over a
smooth quadric in $\PP^3$.
This shows that $P\in X$ is an ordinary double point (node).
\end{proof}

Now we claim that $\rk \Cl(X)=1$.
Indeed, assume that $\rk \Cl(X)>1$. Then we have a non-trivial representation
of $G$ in $\Cl(X)\otimes \QQ$ such that $\rk \Cl(X)^G=1$.
By \cite{atlas} the group $G$ has no non-trivial rational
representations of degree $<10$.
Hence, $\rk \Cl(X)\ge 11$.
Let $F\subset X$ be a prime divisor and let $d:=F\cdot K_X^2$
be its degree.
Consider the $G$-orbit $F_1=F, \dots, F_m$. Then $\sum F_i$ 
is a Cartier divisor on $X$ (because the local Weil divisor class group of every singular point is 
torsion free). Hence, $\sum F_i\sim -rK_X$ for some $r$
and so $md=(2g-2)r=14r$. Since $m$ divides $|G|=660$, $d$ is divisible by
$7$. In particular, $X$ contains no surfaces of degree $\le 6$.
Then by \cite[Cor. 3.12]{Kaloghiros} $\rk \Cl(X)\le 7$, a contradiction.
Therefore, $\rk \Cl(X)=1$.Then by Claim \ref{claim-singular-points-V14} below 
the number of singular points of $X$ is at most $5$.
The contradiction proves the lemma.
\end{proof}

\begin{claim}
\label{claim-singular-points-V14}
Let $X$ be a Gorenstein Fano threefold whose singularities are only \textup(isolated\textup) 
ordinary double points. 
Let $N$ be the number of singular points. Then 
\[
N\le \rk \Cl(X)-\rho(X)+\h^{1,2}(X')-\h^{1,2}(\hat X)
\le \rk \Cl(X)-1+\h^{1,2}(X'), 
\]
where $X'$ is a smoothing of $X$ and $\hat X\to X$ is the blowup of singular points.
\end{claim}
\begin{proof}
Let $D\in |{-}K_X|$ be a general member, let $\tilde X\to X$ be a small 
(not necessarily projective) resolution, 
and let $\tilde D\subset \tilde X$ be the pull-back of $D$.
By the proof of Theorem 13 in \cite{Namikawa-1997} we can write
\begin{multline*}
N\le \dim\mt{Def} (X,D)-\dim\mt{Def} (\tilde X,\tilde D)
=\h^1(X', T_{X'}(-\log D'))-
\\
-\h^1(\tilde X, T_{\tilde X}(-\log \tilde D))=
\textstyle \frac12 \Eu(\tilde X)-\frac12 \Eu(X')=\frac12 \Eu(\hat X)-N-\frac12 \Eu(X'),
\end{multline*}
where $\mt{Def} (X,D)$ (resp. $\mt{Def} (\tilde X,\tilde D)$) 
denotes the deformation space of the pair $(X,D)$ (resp. $(\tilde X,\tilde D)$)
and $(X',D')$ is a general member of the deformation family $\mt{Def} (X,D)$.
Hence, $4N\le \Eu(\hat X)-\Eu(X')$. 
Note that $\rk \Cl(X)=\rho(\hat X)-N$.
Since both $X'$ and $\hat X$ are projective varieties with 
$\Ho^i(X',\OOO_{X'})=\Ho^i(\hat X,\OOO_{\hat X})=0$, we get the disired inequality.
\end{proof}

\begin{lemma}
Under the assumptions of \xref{main-assumption-4} we have $g\le 8$.
\end{lemma}

\begin{proof}
First we consider the case $g=12$.
Then the family of conics 
on $X$ is parameterized by the projective plane $\PP^2$,
see \cite{Kollar2004b}.
By our assumption the induced action of $G$ on $\PP^2$
is trivial. Hence $G$ acts non-trivially on each conic, a contradiction.

Now assume that $g=9$ or $10$.
We claim that in the case $g=9$ the order of $G$ is divisible by $5$
or $11$. 
This follows from Theorem \ref{theorem-Blichfeldt} whenever $G$ has an 
irreducible representation of degree $4$.
Otherwise 
the representation of $G$ in $\Ho^0(X,-K_X)\simeq \CC^{11}$
is either irreducible or has $5$-dimensional irreducible subrepresentation.
By Theorem \ref{theorem-calssification-Fano} and our assumptions
the action of $G$ on $\Ho^{1,2}(X)$ is trivial, so is the action on $\Ho^3(X,\CC)$.
Let $\delta\in G$ be an element of prime order $p\ge 5$.
If $g=9$, then we take $p=5$ or $11$.
Assume that $\delta$ has no fixed points. Then the quotient 
$X/\langle\delta\rangle$ is a smooth Fano threefold.
On the other hand, Fano manifolds are simply-connected,
a contradiction.
Therefore, $\delta$ has at least one fixed point on $X$. 
By the Lefschetz fixed point formula 
we have $\operatorname{Lef}(X,\delta)=4-\dim \Ho^3(X,\CC)=2g-20$.
If $g=9$ or $10$, then $\operatorname{Lef}(X,\delta)\le 0$.
Therefore, the set $\operatorname{Fix}(\delta)$ of $\delta$-fixed points
has positive diminsion. Let $\Phi(X)\subset X$ be the surface 
swept out by lines. Then $\operatorname{Fix}(\delta)\cap \Phi(X)\neq \emptyset$.
Take a point $P\in \operatorname{Fix}(\delta)\cap \Phi(X)$.
Since $X$ is an intersection of quadrics, 
there are at most four lines passing through $P$,
see \cite[Prop. 4.2.2]{Iskovskikh-Prokhorov-1999}. 
The group $\langle\delta\rangle$ cannot interchange these lines.
Hence, there is a $\langle\delta\rangle$-invariant 
line $\ell\subset X$.
Now consider the double projection digram (see \cite{0691.14027},
\cite[Th. 4.3.3]{Iskovskikh-Prokhorov-1999}):
\[
\xymatrix{
&\tilde X\ar@{-->}[r]^{\chi}\ar[dl]_{\sigma}&\tilde X^+\ar[dr]^{\varphi}
\\
X\ar@{-->}[rrr]&&&Y
} 
\]
where $\sigma$ is the blowup of $\ell$ and $\chi$ is a flop.
If $g\ge 9$, then $Y$ is a smooth Fano threefold
and $\varphi$ is the blowup of a smooth curve $\Gamma\subset Y$.
Moreover, 
\begin{enumerate}
\item 
if $g=9$, then $Y\simeq\PP^3$, $\Gamma\subset \PP^3$ 
is a non-hyperelliptic curve of genus $3$ and degree $7$
contained in a unique irreducible cubic surface $F\subset \PP^3$, 
\item
if $g=10$, then $Y=Y_2\subset \PP^4$ is a smooth quadric, 
$\Gamma$ 
is a (hyperelliptic) curve of genus $2$ and degree $7$
contained in a unique irreducible surface $F\subset Y$
of degree $4$. 
\end{enumerate}
Clearly, the above diagram is $\langle\delta\rangle$-equivariant.
Since the linear span of $\Gamma$ coincides with
$\PP^3$ for $g=9$ (resp. $\PP^4$ for $g=10$), the group $\langle\delta\rangle$ 
non-trivially acts on $\Gamma$. 
On the other hand,
the action of $\langle\delta\rangle$ 
on $\Ho^1(\Gamma,\ZZ)\simeq \Ho^3(X,\ZZ)$ is trivial.
This contradicts the Lefschetz fixed point formula.
\end{proof}

Now we are going to finish our treatment of the Gorenstein case. 
It remains to consider two cases: $g=8$ and $g=7$, where 
$X=X_{2g-2}\subset \PP^{g+1}$ is a smooth Fano threefold 
with $\Pic (X)=-K_X\cdot \ZZ$.
Here we need the following result of S. Mukai.

\begin{theorem}[\cite{Mukai-1988}]
\begin{enumerate}
 \item \textup(see also \cite{Gushelcprime1983}\textup)
Let $X=X_{14}\subset \PP^9$ be a smooth Fano threefold of genus $8$
with $\rho(X)=1$. Then $X$ is isomorphic to a linear section of
the Grassmannian $\Gr(2,6)\subset \PP^{14}$
by a subspace of codimension $5$. 
Any isomorphism $X=X_{14}\stackrel{{\sim}\hspace{5pt}}{\longrightarrow} X'=X_{14}'$
of two such smooth sections 
is induced by an isomorphism of the Grassmannian $\Gr(2,6)$.

 \item 
Let $X=X_{12}\subset \PP^8$ be a smooth Fano threefold of genus $7$
with $\rho(X)=1$. Then $X$ is isomorphic to a 
linear section of
the Lagrangian Grassmannian 
$\LGr(4,9)\subset \PP^{15}$ by a subspace of dimension $8$ 
\textup(see Example \xref{example-genus=7}\textup).
Any isomorphism $X=X_{12}\stackrel{{\sim}\hspace{5pt}}{\longrightarrow} X'=X_{12}'$
of two such smooth sections 
is induced by an isomorphism of the Lagrangian Grassmannian 
$\LGr(4,9)$.
\end{enumerate}
\end{theorem}

Consider the case $g=8$.
By the above theorem the group $G$ acts on $\Gr(2,6)$
and on $\PP^{14}=\PP(\wedge^2\CC^5)=\PP(\Ho^0(\Gr(2,6),\TTT^*))$,
where $\TTT$ is tautological rank two vector bundle on $\Gr(2,6)$.
The linear span of $X=X_{12}$ in $\PP^{14}$ is a $G$-invariant $\PP^9$.
Let $\PP^4\subset {\PP^{14}}^*=\PP(\wedge^2{\CC^5}^*)$ 
be the $G$-invariant orthogonal subspace.
The locus of all degenerate skew-forms is the Pfaffian cubic hypersurface 
$Y_3\subset \PP(\wedge^2{\CC^5}^*)$. Put $X_3=Y_3\cap \PP^4$.
Then $X_3\subset \PP^4$ is a $G$-invariant cubic.
Since the variety $X=X_{14}$ is smooth (see Lemma \ref{lemma-Gor-tr}), 
so is our cubic $X_3\subset \PP^4$, see \cite[Prop. A.4]{Kuznetsov2004-e}.
Then by Lemma \ref{lemma-index-Fano} we get 
$G\simeq \PSL_2(11)$, $X^{\mathrm k}_3\subset \PP^4$ is the Klein cubic
and we get Example \xref{example-Klein-cubic}.

Finally consider the case $g=7$.
The group $G$ acts on the Lagrangian Grassmannian 
$\LGr(4,9)\subset \PP^{15}$. 
Let $C:=\LGr(4,9)\cap \PP^{6}$, where $\PP^{6}\subset \PP^{15}$ is 
the subspace orthogonal to $\PP^8$ with respect to the 
$G$-invariant quadratic form on $\PP^{14}$.
Then $C\subset \PP^6$ is a smooth canonical curve of genus $7$ 
\cite[Lemma 3.2]{Iliev2004}. Hence $G\subset \Aut(C)$. On the other hand, 
by the Hurwitz bound we have $|G|\le |\Aut(C)|\le 504$.
Furthermore, the group has an irreducible representation in $\Ho^0(X,-K_X)\simeq \CC^9$.
Hence, $|G|$ is divisible by $9$.
Now it is an easy exercise to show that either $G\simeq\SL_2(8)$ or 
$G$ is contained in the list \eqref{eq-main-cr2}. 
For example, according to Theorem \ref{theorem-transitive-groups} we may assume that
$G$ has no subgroups of index $\le 26$, i.e., of order $\ge 19$.
Hence $234=26\cdot 9\le |G|$.
Now we write the Hurwitz formula for the quotient $\pi: C\to C/G=C'$:
\[
12=2g(C)-2=|G|(2g(C')-2)+ |G|\sum_{i=1}^s (1-1/a_i), 
\]
where $\sum (a_i-1) Q_i$ is the ramification divisor on $C'$.
By the above $a_i\le 18$ for all $i$.
There are only two integer solutions:
$|G|=288$, $(a_1,\dots, a_s)=(2, 3, 8)$ and 
$|G|=504$, $(a_1,\dots, a_s)=(2, 3, 7)$.
In the first case the Sylow $17$-subgroup has index $14$ in $G$, 
a contradiction. In the second case the curve $C$ is unique up to isomorphism
and $G\simeq \PSL_2(8)$, see \cite{Macbeath1965}.
By the construction in Example \ref{example-genus=7}
the threefold $X_{12}$ is uniquely determined by $C$, so $X_{12}=X_{12}^{\mathrm m}$.
This finishes the treatment of the case of Gorenstein $X$. 

\section{Case: $X$ is not Gorenstein}
\label{section-non-Gorenstein}
In this section, as in \S \ref{section-Gorenstein}, we assume that 
$G$ is a simple group which does not admit any embeddings into 
$\Cr_2(\CC)$. We assume $X$ is 
a $G\QQ$-Fano threefold such that $K_X$ is not Cartier.
Let $\Omega\subset \Sing (X)$ be the set of all non-Gorenstein 
points and let $n:=|\Omega|$.
\begin{lemma}
\label{lemma-cyclic-quotient}
In the above assumptions the group 
$G$ transitively acts on $\Omega$, $n\ge 9$, and each point $P\in \Omega$ 
is a cyclic quotient singularity of index $2$.
\end{lemma}

\begin{proof}
Let $\Omega=\Omega_1\cup\dots\cup\Omega_m$ be the orbit decomposition, and let $n_i:=|\Omega_i|$.
For a point $P_i\in \Omega_i$, let $Q_{ij}\in \B$, $j=1,\dots,l_i$ be ``virtual'' 
points in the basket over $P_i$ and let $r_{ij}$ be the index of $Q_{ij}$. 
The orbifold Riemann-Roch and Myaoka-Bogomolov inequality give us
(see \cite{Kawamata-1992bF}, \cite{KMMT-2000})
\footnote{From \cite{KMMT-2000} we have the inequality
$\sum (r-1/r)\le 24$. The strict inequality follows from 
the proof in \cite{Kawamata-1992bF} because $\rho(X)^G=1$.
I would like to thank Professor Y. Kawamata for pointing me out this fact.} 
\begin{equation}
\label{equation-RR-inequality}
24 > 
\sum_{i=1}^m n_i \sum_{j=1}^{l_i}\left(r_{ij}-\frac 1{r_{ij}}\right)\ge 
\frac 32 \sum_{i=1}^m n_i.
\end{equation}
By Theorem \ref{theorem-transitive-groups} and our assumptions we have $n_1,\dots, n_m\ge 7$.

Assume that $P_1\in X$ is not a cyclic quotient singularity.
Then over each $P_i\in \Omega_1$
there are at least two virtual points $Q_{ij}$, i.e, $l_1>1$.
By \eqref{equation-RR-inequality} we have
\[
24 > 
n_1 \sum_{j=1}^{l_1}\left(r_{1j}-\frac 1{r_{1j}}\right)\ge 
7 \sum_{j=1}^{l_1}\left(r_{1j}-\frac 1{r_{1j}}\right). 
\]
There is only one possibility: $l_1=2$, $n_1=7$, and $r_{11}=r_{12}=2$.
In this case, by the classification \cite[Th. 6.1]{Reid-YPG1987} 
the point $P_1\in X$ is of type $\{xy+\phi(z^2,t)\}/\muu_2(1,1,1,0)$,
where $\ord \phi(0,t)=2$, 
or $\{x^2+y^2+\phi(z,t)\}/\muu_2(0,1,1,1)$ (because the ``axial multiplicity'' 
is equal to $2$).

By Theorem \ref{theorem-transitive-groups} we have
$G\simeq \A_7$ and $G_P\simeq \A_6$.
As in the proof of Lemma \ref{lemma-fixed-point}
we have an embedding $\tilde G_P \subset \GL(T_{P^\sharp,U^\sharp})$,
where $\dim T_{P^\sharp,U^\sharp}=4$ 
and $\tilde G_P$ is a central 
extension of $G_P$ by $\muu_2$. The action of $\tilde G_P$
preserves the tangent cone $TC_{P^\sharp,U^\sharp}\subset T_{P^\sharp,U^\sharp}$
which is given by a quadratic form of rank $\ge 2$.
Since $G_P\simeq \A_6$ cannot act non-trivially on a 
smooth quadric in $\PP^3$, $\rk q\neq 4$.
Hence, $\rk q=2$ or $3$ and the representation of $\tilde G_P$ in 
$T_{P^\sharp,U^\sharp}\simeq \CC^4$ is reducible:
the singular locus of $TC_{P^\sharp,U^\sharp}$ is a 
$\tilde G_P$-invariant linear subspace.
On the other hand, $\tilde G_P\simeq \A_6$ has no faithful representations of 
degree $\le 3$ 
(see, e.g., Theorem \ref{theorem-classification-3} or \cite{atlas}), a contradiction.

Therefore, all the points in $\Omega$ are cyclic quotient singularities.
Then \eqref{equation-RR-inequality} can be rewritten as follows:
\begin{equation}
\label{equation-RR-inequality-2}
24 > 
\sum_{i=1}^m n_i \left(r_{i}-\frac 1{r_{i}}\right)\ge 
\frac 32 \sum_{i=1}^m n_i,
\end{equation}
where $r_i$ is the index of the point $P_i\in \Omega_i$. 
Assume that $n_1\le 8$, then by Theorem \ref{theorem-transitive-groups}
$G\simeq \A_n$ with $n=7$ or $8$, and $G_P\simeq \A_{n-1}$.
As above $\tilde G_P \subset \GL(T_{P^\sharp,U^\sharp})$,
where $\dim T_{P^\sharp,U^\sharp}=3$ (because $U^\sharp$ is smooth)
and $\tilde G_P$ is a central 
extension of $G_P$ by $\muu_{r_1}$.
Clearly, the representation 
$\tilde G_P$ in $\GL(T_{P^\sharp,U^\sharp})$
is irreducible. Hence $\muu_{r_1}$ acts on $T_{P^\sharp,U^\sharp}$
by scalar multiplication. As in the proof of Lemma \ref{lemma-fixed-point},
by the classification of terminal singularities
(Terminal Lemma) \cite{Reid-YPG1987}, we have $r_1=2$.
But then the group $\tilde G_P$ has no non-trivial representations in 
$\CC^3$ by Theorem \ref{theorem-classification-3}.
The contradiction shows that $n_1\ge 9$ and, by symmetry, $n_i\ge 9$
for all $i$. Then by \eqref{equation-RR-inequality-2}
we have $24> 9m \cdot 3/2$. Hence $m=1$, i.e., $\Omega$ consists of one orbit. Further,
$24> 9(r_i-1/r_i)$. Hence $r_i=2$ for all $i$.
\end{proof}

\begin{lemma}
\begin{enumerate}
\item 
$Z(G_P)=\{1\}$, $Z(\tilde G_P)=\muu_2$.
\item 
The representation of $\tilde G_P$ in $T_{P^\sharp, U^\sharp}\simeq\CC^3$
is irreducible.
\item 
The action of $\tilde G_P$ on $T_{P^\sharp, U^\sharp}\simeq\CC^3$
is primitive.
\item
The only possible case is 
$G\simeq\PSL_2(11)$, $n=11$, $G_P\simeq\A_5$.
\end{enumerate}
\end{lemma}
\begin{proof}
(i) follows from the explicit description of groups $G_P$
in Theorem \ref{theorem-transitive-groups}.

(ii)
Assume that $T_{P^\sharp, U^\sharp}=T_1\oplus T_2$.
Then the kernel of the homomorphism 
$\tilde G_P\to \GL(T_2)$ is contained into $Z(\tilde G_P)$.
Hence, $\tilde G_P\to \GL(T_2)$ is injective and so
$G_P$ effectively acts on $\PP^1$, a contradiction.

(iii) 
Assume the converse. Then there is an abelian subgroup
$\tilde A\subset \tilde G_P$ such that $\tilde G_P/\tilde A\simeq
\A_3$ or $\Sy_3$.
Hence there is an abelian subgroup
$A\subset G_P$ such that $G_P/A\simeq \tilde G_P/\tilde A\simeq
\A_3$ or $\Sy_3$. In particular, $G_P$ is not simple and its 
order is divisible by $3$.
Thus by Theorem \ref{theorem-transitive-groups}
there are only three possibilities: 
$G\simeq \SL_3(3)$, $\PSL_2(13)$, and $\SL_4(2)$.

In the case $G\simeq \PSL_2(13)$ the group 
$G_P\simeq \muu_{13}\rtimes \muu_6$ 
has no surjective homomorphisms to $\Sy_3$.
So, $G_P/A\simeq \A_3$ and $A\simeq \muu_{26}$.
On the other hand, $G_P$ contains no elements of order $26$,
a contradiction.
Consider the case $G\simeq \SL_3(3)$. Then $G_P\supset \GL_2(3)$. 
Since $\GL_2(3)/Z(\GL_2(3))\simeq \Sy_4$,
for $A\cap \GL_2(3)$ we have only one possibility:
it is a group of order $8$.
But the group $A\cap \GL_2(3)$ is not abelian, a contradiction.
Finally, in the case $G\simeq \SL_4(2)$ the group 
$\SL_3(2)\subset G_P$ is simple of order $168$,
a contradiction. 

(iv) Follows from Theorem \ref{theorem-classification-3}.
\end{proof}

From now on we assume that $G\simeq \PSL_2(11)$ and $n=11$.

\begin{lemma}
$\dim |{-}K_X|> 0$.
\end{lemma}
\begin{proof}
By \cite{Kawamata-1992bF} we have $-K_X\cdot c_2(X)=24-3n/2$.
Hence by the orbifold Riemann-Roch (see \cite{Reid-YPG1987})
\begin{multline*}
\dim |{-}K_X|=\frac 12 (-K_X)^3-\frac1{12}K_X\cdot c_2(X)
+\sum_{P\in \Omega} c_P(-K_X)=
\\
=\frac 12 (-K_X)^3 +2-\frac {n}4=
\frac 12 (-K_X)^3 -\frac {3}4.
\end{multline*}
Put $\dim |{-}K_X|=l$. Then 
$(-K_X)^3=2l+3/2$.
In particular, $l\ge 0$ and $|{-}K_X|\neq \emptyset$.
Assume that $\dim |{-}K_X|= 0$. Then $(-K_X)^3=3/2$. 
Let $S\in |{-}K_X|$ be (a unique) member.
By Lemma \ref{lemma-hyp-sect} the surface $S$ is reducible and $G$
transitively acts on its components.
Write $S=\sum _{i=1}^m S_i$.
Then $m(-K_X)^2\cdot S_i=(-K_X)^3=3/2$.
Since $2(-K_X)^2\cdot S_i$ is an integer, we have $m\le 3$, a contradiction.
\end{proof}

\begin{lemma}
\label{lemma-LC-property-linear-systems}
The pair $(X,|{-}K_X|)$ is canonical.
\end{lemma}

\begin{proof}
Put $\HHH:=|{-}K_X|$.
By Corollary \ref{corollary-fixed-components-linear-systems} 
the linear system $\HHH$ has no fixed components.
We apply a $G$-equivariant version of a construction 
\cite[\S 4]{Alexeev-1994ge}.
Take $c$ so that the pair $(X,c\HHH)$ is canonical but not terminal.
By our assumption $0<c<1$.
Let $f\colon (\tilde X,c \tilde \HHH)\to (X,c\HHH)$ 
be a $G$-equivariant $G\QQ$-factorial terminal modification (terminal model).
We can write 
\[
\begin{array}{lll}
K_{\tilde X}+c\tilde \HHH&=&f^*(K_X+c\HHH),
\\
K_{\tilde X}+\tilde \HHH+\sum a_iE_i&=&f^*(K_X+\HHH)\sim 0, 
\end{array}
\]
where $E_i$ are $f$-exceptional divisors and $a_i>0$.
Run $(\tilde X,c \tilde \HHH)$-MMP:
\[
(\tilde X,c \tilde \HHH) \dashrightarrow (\bar X,c \bar \HHH).
\]
As in \ref{main-reduction} $\bar X$ is a Fano threefold with 
$G\QQ$-factorial terminal singularities and $\rho(\bar X)^G=1$.
We also have $0\sim K_{\bar X}+\bar \HHH+\sum a_i\bar E_i$.
Here $\sum a_i\bar E_i$ is a non-trivial effective invariant divisor
such that $-(K_{\bar X}+\sum a_i\bar E_i)\sim \bar \HHH$
is ample. This contradicts Lemma \ref{lemma-hyp-sect}. 
\end{proof}

\begin{lemma}
\label{lemma-index2-big}
The image of the \textup($G$-equivariant\textup) rational map 
$\phi: X\dashrightarrow \PP^l$ given by 
the linear system $|{-}K_X|$ 
is three-dimensional.
\end{lemma}

\begin{proof}
Let $Y:=\phi(X)$. 
Since $X$ is rationally connected, $G$ acts trivially on $Y$.
This contradicts Lemma \ref{lemma-Gorenstein-invariant-pencil}.
\end{proof}
Recall that non-Gorenstein points $P_1,\dots, P_{11}$ of $X$ are of type $\frac12(1,1,1)$.
Let $f: \tilde X\to X$ be blow up of $P_1,\dots, P_{11}$ 
 and let $E=\sum E_i$ be the exceptional divisor, where $E_i=f^{-1}(P_i)$.
Then $\tilde X$ is smooth over $P_i$, it has at worst Gorenstein terminal 
singularities, $E_i\simeq \PP^2$, and $\OOO_{E_i}(-K_{\tilde X})=\OOO_{\PP^2}(1)$.
Put $\HHH:=|{-}K_X|$ and let $\tilde \HHH$ be the birational transform.
Since the pair $(X,\HHH)$ is canonical, we have
\[
K_{\tilde X}+\tilde \HHH\sim f^*(K_X+\HHH)\sim 0. 
\]
Hence, $|{-}K_{\tilde X}|=\tilde \HHH$.

\begin{lemma}
The linear system $\tilde \HHH$ is base point free.
\end{lemma}

\begin{proof}
Note that the restriction $\tilde \HHH|_{E_i}=|{-}K_{\tilde X}||_{E_i}$ 
is a (not necessarily complete) linear system
of lines on $E_i\simeq\PP^2$. 
Since this linear system is $G_{P_i}$-invariant,
where $G_{P_i}\simeq \A_5$,
it is base point free. Hence $\Bs\tilde \HHH\cap E_i=\emptyset$.
In particular, this implies that for any curve 
$C\subset \Bs\tilde \HHH$, we have 
$E_i\cdot C=0$ and so $\tilde \HHH\cdot C=\HHH\cdot f(C)>0$.
Hence, $\tilde \HHH$ is nef. 
By Lemma \ref{lemma-index2-big} it is big.
Then the assertion follows by Lemma \ref{lemma-base-point-free}.
\end{proof}

Now we are in position to finish the proof of Theorems
\ref{theorem-main-1} and \ref{theorem-main-2}.
Note that the divisors $E_i$ are linear independent elements 
of $\Pic (\tilde X)$. Hence, $\rho(\tilde X)>11$.
If $\tilde X$ is a Fano threefold, then by
Theorem \ref{theorem-Namikawa} and Remark
\ref{remark-Namikawa} there is a smoothing $\tilde X_t$
with $\rho(\tilde X_t)>11$.
This contradicts the classification of smooth Fano threefolds with 
$\rho>1$ \cite{Mori1981-82}.
By Lemma \ref{lemma-very-ample} the linear system $|{-}K_{\tilde X}|$
determines a birational contraction $\varphi:\tilde X\to \bar X=\bar X_{2g-2}\subset \PP^{g+1}$
whose image is an anticanonically embedded Fano threefold with at worst canonical singularities.
Here $g$ is the genus of $\bar X$ (see \ref{main-assumption-2}). 
By Lemma \ref{lemma-quadrics}  the variety $\bar X_{2g-2}\subset \PP^{g+1}$ is an intersection of quadrics.
In particular, $g\ge 5$.
Since $\rho(\tilde X)^G=2$, 
$\rho(\bar X)^G=1$.
Let $\bar E_i:=\varphi(E_i)$ and $\bar E:=\varphi(E)$.

\begin{claim}\label{claim-cl-Z}
The group $\Cl(X)^G$ is generated by $-K_X$. 
\end{claim}
\begin{proof}
Assume that $\Cl(X)^G$  contains a torsion element, say  $T$.
Then $2T$ is Cartier and so $2T\sim 0$
(because $\Pic(X)$ is torsion free).
As above, let $\Omega\subset X$  be the set of points where $K_X$ 
is not Cartier. 
Since $G$ acts on $\Omega$ transitively, $T$ is not Cartier at all points $P\in \Omega$.
On the other hand, by the orbifold  Riemann-Roch (see \cite{Reid-YPG1987}) and Kawamata-Viehweg vanishing theorem we have 
\[
0=\chi(\OOO_X(T))=1+\sum _{P\in \Omega} c_P(T)= 1- \frac{|\Omega|}{8},
\]
where $c_P(T)=-1/8$ for all $P\in \Omega$ (because
this $P$ is a cyclic quotient of type $\frac12(1,1,1)$).
This gives us $|\Omega|=8$, a contradiction.

Therefore,  $\Cl(X)^G\simeq \ZZ$. Let $A$ be the ample generator of $\Cl(X)^G$ and let 
$-K_X=qA$ for some $q\in \ZZ_{>0}$. Since $K_X$ is not Cartier, $q>2$. Again by the orbifold  Riemann-Roch
\begin{multline*}
\label{eq-RR}
\chi(\OOO_X(-A))=
-\frac{A^3}{12}(q-1) (q-2)
-\frac{A\cdot c_2}{12} +\sum_{P\in \Omega} c_P(-A)+ 
1 <
\\
<\sum_{P\in \Omega} c_P(-A)+ 
1 =-\frac{11}{8}+ 1<0.
\end{multline*}
On the other hand, by the Kawamata-Viehweg vanishing theorem $\chi(\OOO_X(-A))=0$, a contradiction.
\end{proof}

Take a general member $\bar H\in |{-}K_{\bar X}|$.
By Bertini's theorem $\bar H$ is a K3 surface with at worst 
Du Val singularities. 
Put $C_i:=\bar E_i\cap \bar H$.

\begin{claim}
$C_1,\dots,C_{11}$ are disjointed smooth rational curves contained into the 
smooth locus of $\bar H$. 
\end{claim}

\begin{proof}
Since $\bar H$ is Cartier, the number 
$\bar E_i\cdot \bar E_j\cdot \bar H$, where $1\le i,\, j\le 11$,
is well-defined and coincides with the intersection number $C_i\cdot C_j$
of curves $C_i:=\bar E_i\cap \bar H$ and $C_j:=\bar E_j\cap \bar H$
on $\bar H$.
Clearly, the numbers 
$C_i^2=\bar E_i^2\cdot \bar H$ for $1\le i\le 11$ do not depend on $i$.
Since the action of $G$ on $\{\bar E_i\}$ is doubly transitive \cite{atlas}, the numbers 
$C_i\cdot C_j=\bar E_i\cdot \bar E_j\cdot \bar H$ for $1\le i\neq j\le 11$ also 
do not depend on $i$, $j$.

Since $(-K_{\tilde X})^2\cdot E_i=1$, the surfaces $\bar E_i$ are planes in $\PP^{g+1}$
and every $C_i$ is a line on $\bar E_i$.
If $C_i\cdot C_j>0$ for some $i\neq j$, then
$\bar E_i\cap \bar E_j$ is a line. Since $G$ acts doubly transitive on 
$\{\bar E_i\}$, the intersection $\bar E_i\cap \bar E_j$ is a line
for all $i\neq j$. Hence, the linear span of $\bar E_1\cup \bar E_2\cup \bar E_3$
is a three-dimensional projective subspace $\PP^3\subset \PP^{g+1}$.
In this case, $\bar X\cap \PP^3$ cannot be an intersection of 
quadrics. This contradicts Lemma \ref{lemma-quadrics}.

Thus we may assume that $C_i\cdot C_j=0$ for all $i\neq j$.
By the Hodge index theorem
$C_k^2\le 0$ for all $k$.
If $C_1^2=0$, then $C_1$ is a nef $\QQ$-Cartier divisor on a
K3 surface with at worst Du Val singularities.
By the cone theorem,
for some $m$, the linear system $|mC_1|$
determines an elliptic fibration $\psi: \bar H\to \PP^1$ and 
all the curves $C_k$ are degenerate fibers of $\psi$.
Let $\mu: \hat H\to \bar H$ be the minimal resolution,
let $F_k:=\mu^{-1}(C_k)$ be the degenerate fiber corresponding to 
$C_k$, and let $\hat C_k$ be the proper transform of $C_k$.
Then $\hat H$ is a smooth K3 surface.
Since $C_k$ is smooth, $\hat C_k\cdot (F_k-\hat C_k)=1$.
Using Kodaira's classification of degenerate fibers of
elliptic fibrations we see that $F_k$ has at least 
three components. 
But then $\rho(\hat H)\ge 23$, a contradiction.

Therefore, $C_k^2<0$ for all $k$. In particular, $\rk \Cl(\bar H)\ge 12$.
If $\bar H$ is singular at a point on $C_k$, then, 
as above, considering the minimal resolution $\mu: \hat H\to \bar H$ 
one can show that $\rho(\hat H)\ge 23$, a contradiction.  
Hence $\bar H$ is smooth near 
$C_k$. So, all the $C_k$ are $(-2)$-curves contained into 
the smooth part of $\bar H$. 
\end{proof}

Clearly, fibers of $\varphi$ meet $\sum E_i$ (otherwise 
$\varphi$ is an isomorphism near $E_i$ and then $\rho(\bar X)^G>1$).
Since $E_i\simeq \PP^2$, $\varphi$ cannot contract divisors to points.
Assume that $\varphi$ contracts divisors $D_l$ to curves $\Gamma_l$.
Then $\Gamma_l\subset E_i$ for some $i$. Since $\varphi$ is $K$-trivial,
$\bar X$ is singular along $\Gamma_l$ and
$\bar H$ is singular at point $\Gamma_l\cap \bar H$.
Since $\Gamma_l\cap \bar H\subset C_i$, we get a contradiction
with the above claim.

Therefore $\varphi$ does not contract any divisors, i.e., 
it contracts only a finite number of curves.
Then $\bar X$ is a Fano threefold with 
Gorenstein terminal (but not $G\QQ$-factorial) singularities.
Consider the following diagram (cf. \cite[Ch. 4]{Iskovskikh-Prokhorov-1999}):
\[
\xymatrix{
&\tilde X\ar[dr]^{\varphi}\ar[dl]_{f}\ar@{-->}[rr]^{\chi}&&X^+\ar[dr]^{f^+}\ar[dl]_{\varphi^+}
\\
X&&\bar X&&Y
}
\]
Here $\chi$ is a $G$-equivariant flop, $\varphi^+$ is a small modification,
and $f^+$ is a $K$-negative $G$-equivariant $G$-extremal contraction.
As in \ref{main-reduction} we may assume that $Y$ is $G\QQ$-Fano threefold with 
$\rho(Y)^G=1$. 
Let $E^+=\sum E_i^+\subset X^+$ be the proper transform of $E=\sum E_i$.
Recall that $G\simeq \PSL_2(11)$. We can write 
\[
\begin{array}{l}
-K_{\tilde X}^3=-K_{X^+}^3=-K_{\bar X}^3=2g-2,
\\[8pt]
(-K_{\tilde X})^2\cdot E=(-K_{X^+})^2\cdot E^+=(-K_{\bar X})^2\cdot \bar E=11,
\\[8pt]
-K_{\tilde X}\cdot E^2=-K_{X^+}\cdot E^+{}^2=-K_{\bar X}\cdot \bar E^2=-22.
\end{array}
\]
Let $D:=\sum D_i$ be the $f^+$-exceptional divisor. By Claim \ref{claim-cl-Z}
we have
$D\sim -\alpha K_{X^+}-\beta E^+$ for some $\alpha$, $\beta\in \NN$.
Therefore, 
\[
\begin{array}{l}
(-K_{\bar X})^2\cdot D=(2g-2)\alpha-11 \beta,
\\[8pt]
-K_{\bar X}\cdot D^2=(2g-2)\alpha^2-22\alpha\beta -22\beta^2.
\end{array}
\]

Assume that $Y$ is not Gorenstein.
Then $Y$ is of the same type as $X$. In particular, $Y$ has 11 
cyclic quotient singular points of index $2$.
In this case $D$ has exactly $11$ components and 
\begin{equation}
\label{equation-intersections-last-S}
(-K_{\bar X})^2\cdot D=11,
\qquad 
-K_{\bar X}\cdot D^2=-22.
\end{equation}
In particular, either $g-1$ or $\alpha$ is divisible by $11$.
Assume that $g-1=11k$, $k\in \NN$. 
Then the above equalities can be rewritten as follows:
\[
\begin{array}{l}
\beta=2 k\alpha-1,
\\[8pt]
0=-1-k\alpha^2+\alpha\beta +\beta^2.
\end{array}
\]
Eliminating $\beta$ we get
\[
0=-1-k\alpha^2+\alpha(2 k\alpha-1) +(2 k\alpha-1)^2=
(\alpha+ 4 k)(k\alpha-1).
\]
Since $\alpha,\, k>0$ we get 
$k=1$ and $g=12$. Hence $\dim \Ho^0(\tilde X,-K_{\tilde X})=14$
and so $\dim \Ho^0(\tilde X,-K_{\tilde X})^G\ge 2$ (because the degrees 
of irreducible representations of $G=\PSL_2(11)$ are 
$1$, $5$, $5$, $10$, $10$, $11$, 
$12$, $12$). This contradicts Lemma 
\ref{lemma-Gorenstein-invariant-pencil}.
Therefore, $\alpha =11k$, $k\in \NN$. Then, as above,
\[
\begin{array}{l}
\beta=2(g-1)k- 1,
\\[8pt]
0=-1-11(g-1)k^2+11k\beta +\beta^2.
\end{array}
\]
Thus
\begin{multline*}
0=-1-11(g-1)k^2+11k(2(g-1)k- 1) +(2(g-1)k- 1)^2=
\\
=(11+ 4(g-1))((g-1)k-1). 
\end{multline*}
Since $g>2$ (see Lemma \ref{lemma-very-ample}) we have a contradiction.

Finally assume that $Y$ is Gorenstein.
By the 
results of \S \ref{section-Gorenstein} either
$Y\simeq X^{\mathrm k}_3\subset \PP^4$ or $Y\simeq X_{14}^{\mathrm a}\subset \PP^{9}$.
In particular, $Y$ is smooth.
If $\dim f^+(D)=0$, then $f^+$ is just blowup of points $Q_1,\dots,Q_l\in Y$ \cite{Cutkosky-1988}. 
Note that $\rk \Cl(X^+)=\rk \Cl(\tilde X)\ge 12$, so $l\ge 11$.
But then $-K_{X^+}^3= -K_Y^3- 8 l<0$, a contradiction. 
Therefore $f^+(D)$ is a (reducible) curve $\Gamma=\sum_{i=1}^l \Gamma_i$.
Here again $l\ge 11$. Write
\[
\varphi^{+*}(-K_{\bar X})=-K_{X^+}=f^{+*}(-K_Y)-D. 
\]
Since the linear system $|{-}K_{\bar X}|$ is very ample, the map 
$Y \dashrightarrow \bar X=\bar X_{2g-2}\subset \PP^{g+1}$ is given by 
a subsystem $f^+_*|{-}K_{X^+}|$ of the linear system $|{-}K_Y|$
consisting of elements passing through $\Gamma$. 
On the other hand, in the case $Y\simeq X_{14}^{\mathrm a}\subset \PP^{9}$,
the representation of $G$ in $\Ho^0(Y,-K_Y)$ is irreducible 
(see Example \ref{example-V14}). Therefore, $Y\simeq X^{\mathrm k}_3\subset \PP^4$.
Moreover, $\dim |{-}K_{\bar X}|< \dim |{-}K_Y|= 14$. 
By Lemma \ref{lemma-hyp-sect} the representation of $G$ in 
$\dim \Ho^0(Y,-K_{Y})^G$ has no trivial subrepresentations.
Since $g\ge 5$ we have only one possibility:
$\dim |{-}K_{\bar X}|=9$, $g=8$.
Further, by Claim \ref{claim-cl-Z} the group  
 $\Cl(X^+)^G$ is a free $\ZZ$-module generated by $-K_{X^+}$ and $E$.
On the other hand,  $\Cl(X^+)^G$ is generated by $\frac12 f^{+*}(-K_Y)$ and $D$.
Hence $\beta=2$, i.e., $D\sim -\alpha K_{X^+}-2E^+$.
Similar to \eqref{equation-intersections-last-S} we have (see e.g. \cite{Kaloghiros})
\[
\begin{array}{rllll}
-K_Y\cdot \Gamma-2p_a(\Gamma)+2&=&(-K_{\bar X})^2\cdot D&=&14\alpha-22,
\\[8pt]
2p_a(\Gamma)-2&=&-K_{\bar X}\cdot D^2&=&14\alpha^2-44\alpha -88.
\end{array}
\]
Thus $\deg \Gamma=-\frac12 K_Y\cdot \Gamma=7\alpha^2 - 15\alpha - 55\ge 45$.
On the other hand, $\Gamma$ is a scheme intersection of members of the linear system
$f^+_*|{-}K_{X^+}|\subset |{-}K_Y|$ (because $|{-}K_{X^+}|$ is base point free).
Hence, $\deg \Gamma\le 12$, a contradiction.
This finishes our proof of Theorems
\ref{theorem-main-1} and \ref{theorem-main-2}.

%\bibliography{my_ref,specific/inv_groups}
%%\bibliography{specific/inv}
%%\bibliography{/home/yuri/bib/specific/inv}

\begin{thebibliography}{KMMT00}

\bibitem[Adl78]{Adler1978}
A. Adler.
\newblock On the automorphism group of a certain cubic threefold.
\newblock {\em Amer. J. Math.}, 100(6):1275--1280, 1978.


\bibitem[AR96]{AdlerRamanan1996}
A. Adler and S. Ramanan.
\newblock {\em Moduli of abelian varieties}, {\em Lecture Notes
 in Mathematics}  1644.
\newblock Springer-Verlag, Berlin, 1996.



\bibitem[AW97]{Abramovich-Wang}
D. Abramovich and J. Wang.
\newblock {Equivariant resolution of singularities in characteristic 0.}
\newblock {\em Math. Res. Lett.}, 4(2-3):427--433, 1997.

\bibitem[Ale94]{Alexeev-1994ge}
V. Alexeev.
\newblock General elephants of {${\bf Q}$}-{F}ano 3-folds.
\newblock {\em Compositio Math.}, 91(1):91--116, 1994.


\bibitem[ACGH85]{Arbarello1985}
E.~Arbarello, M.~Cornalba, P.~A. Griffiths, and J.~Harris.
\newblock {\em Geometry of algebraic curves. {V}ol. {I}},  {\em
  Grundlehren der Mathematischen Wissenschaften} [{\em Fundamental Principles of
  Mathematical Sciences}]  {267}.
\newblock Springer-Verlag, New York, 1985.

\bibitem[Bli17]{Blichfeldt}
H.~F. Blichfeldt.
\newblock {\em Finite collineation groups}.
\newblock The Univ. Chicago Press, Chicago, 1917.

\bibitem[Bra67]{Brauer-1967}
R. Brauer.
\newblock \"{U}ber endliche lineare {G}ruppen von {P}rimzahlgrad.
\newblock {\em Math. Ann.}, 169:73--96, 1967.

\bibitem[CCN{\etalchar{+}}85]{atlas}
J.~H. Conway, R.~T. Curtis, S.~P. Norton, R.~A. Parker, and R.~A. Wilson.
\newblock {\em {Atlas of finite groups. Maximal subgroups and ordinary
 characters for simple groups. With comput. assist. from J. G. Thackray.}}
\newblock {Oxford: Clarendon Press. XXXIII, 252 p.}, 1985.

\bibitem[CG72]{Clemens-Griffiths}
H. Clemens and P. Griffiths.
\newblock The intermediate {J}acobian of the cubic threefold.
\newblock {\em Ann. of Math. (2)}, 95:281--356, 1972.

\bibitem[Cut88]{Cutkosky-1988}
S. Cutkosky.
\newblock Elementary contractions of {G}orenstein threefolds.
\newblock {\em Math. Ann.}, 280(3):521--525, 1988.

\bibitem[DI06]{Dolgachev-Iskovskikh}
I. Dolgachev and V. Iskovskikh.
\newblock Finite subgroups of the plane cremona group.
\newblock to appear in {\em Algebra, Arithmetic, and Geometry In Honor of Yu. I. Manin},
Progr. Math., 269--270, Birkh\"auser, 2009.

\bibitem[DM96]{Dixon-Mortimer}
J.~D. Dixon and B. Mortimer.
\newblock {\em {Permutation groups.}}
\newblock {Graduate Texts in Mathematics. 163. New York, NY: Springer-Verlag.
 xii, 346 p.}, 1996.

\bibitem[GH78]{Griffiths1978}
P. Griffiths and J. Harris.
\newblock {\em {Principles of algebraic geometry.}}
\newblock {Pure and Applied Mathematics. A Wiley-Interscience Publication. New
 York etc.: John Wiley \&amp; Sons. XII, 813 p.}, 1978.

\bibitem[Gus83]{Gushelcprime1983}
N.~P. Gushel{\cprime}.
\newblock Fano varieties of genus {$8$}.
\newblock {\em Uspekhi Mat. Nauk}, 38(1(229)):163--164, 1983.

\bibitem[IM04]{Iliev2004}
A. Iliev and D. Markushevich.
\newblock Elliptic curves and rank-2 vector bundles on the prime {F}ano
 threefold of genus 7.
\newblock {\em Adv. Geom.}, 4(3):287--318, 2004.

\bibitem[Isk77]{Iskovskih1977a}
V.~A. Iskovskih.
\newblock Fano threefolds. {I}.
\newblock {\em Izv. Akad. Nauk SSSR Ser. Mat.}, 41(3):516--562, 717, 1977.

\bibitem[Isk80]{Iskovskikh-1980-Anticanonical}
V.~A. Iskovskikh.
\newblock {Anticanonical models of three-dimensional algebraic varieties.}
\newblock {\em J. Sov. Math.}, 13:745--814, 1980.

\bibitem[Isk90]{0691.14027}
V.~A. Iskovskikh.
\newblock {A double projection from a line on {F}ano threefolds of the first
 kind.}
\newblock {\em Math. USSR, Sb.}, 66:265--284, 1990.

\bibitem[IM71]{Iskovskikh1971}
V.~A. Iskovskikh and Yu.~I. Manin.
\newblock {Three-dimensional quartics and counterexamples to the L\"uroth
  problem.}
\newblock {\em Math. USSR-Sb.}, 15:141--166, 1971.

\bibitem[IP99]{Iskovskikh-Prokhorov-1999}
V.~A. Iskovskikh and Yu. Prokhorov.
\newblock {\em Fano varieties. {A}lgebraic geometry. {V}.}, {\em
 Encyclopaedia Math. Sci.} ~47.
\newblock Springer, Berlin, 1999.

\bibitem[IP96]{Iskovskikh-Pukhlikov-1996}
V.~A. Iskovskikh and A.~V. Pukhlikov.
\newblock Birational automorphisms of multidimensional algebraic manifolds.
\newblock {\em J. Math. Sci.}, 82(4):3528--3613, 1996.

\bibitem[JR06]{jahnke-Radloff-2006arx}
P. Jahnke and I. Radloff.
\newblock Terminal {F}ano threefolds and their smoothings.
\newblock preprint arXiv{:} 0601769, 2006.

\bibitem[Kal11]{Kaloghiros}
A.-S. Kaloghiros.
\newblock The defect of {F}ano 3-folds.
\newblock {\em J. Algebraic Geom.}, 20(1):127--149, 2011,
erratum available at \url{http://www.dpmms.cam.ac.uk/~asack2}.

\bibitem[Kaw92]{Kawamata-1992bF}
Y. Kawamata.
\newblock Boundedness of {$\mathbf{Q}$}-{F}ano threefolds.
\newblock In {\em Proceedings of the International Conference on Algebra, Part
 3 (Novosibirsk, 1989)}, {\em Contemp. Math.} \textbf{131}, 439--445,
 Providence, RI, 1992. Amer. Math. Soc.

\bibitem[Kaw97]{Kawamata-1997-Adj}
Yujiro Kawamata.
\newblock Subadjunction of log canonical divisors for a subvariety of
  codimension {$2$}.
\newblock In {\em Birational algebraic geometry (Baltimore, MD, 1996)},  
{\em Contemp. Math.}   \textbf{207}, 79--88. Amer. Math. Soc., Providence, RI,
  1997.

\bibitem[Kaw07]{Kawakita2007}
M. Kawakita.
\newblock Inversion of adjunction on log canonicity.
\newblock {\em Invent. Math.}, 167:129--133, 2007.

\bibitem[KMMT00]{KMMT-2000}
J. Koll{\'a}r, Y. Miyaoka, S. Mori, and H. Takagi.
\newblock Boundedness of canonical {$\mathbf{Q}$}-{F}ano 3-folds.
\newblock {\em Proc. Japan Acad. Ser. A Math. Sci.}, 76(5):73--77, 2000.

\bibitem[KS04]{Kollar2004b}
J. Koll{\'a}r and F.-O. Schreyer.
\newblock Real {F}ano 3-folds of type {$V\sb {22}$}.
\newblock In {\em The Fano Conference}, 515--531. Univ. Torino, Turin,
 2004.

\bibitem[Kol92]{Utah}
J.~Koll{\'a}r, editor.
\newblock {\em Flips and abundance for algebraic threefolds}.
\newblock Soci\'et\'e Math\'ematique de France, Paris, 1992.
\newblock Papers from the Second Summer Seminar on Algebraic Geometry held at
  the University of Utah, Salt Lake City, Utah, August 1991, Ast\'erisque No.
  211 (1992).

\bibitem[Kuz04]{Kuznetsov2004-e}
A.~G. Kuznetsov.
\newblock {Derived categories of cubic and $V_{14}$ threefolds.}
\newblock {\em Proceedings of the Steklov Institute of Mathematics},
 246:171--194, 2004.


\bibitem[Lin71]{Lindsey-1971-lg6}
J.~H. Lindsey, II.
\newblock Finite linear groups of degree six.
\newblock {\em Canad. J. Math.}, 23:771--790, 1971.

\bibitem[Mac65]{Macbeath1965}
A.~M. Macbeath.
\newblock On a curve of genus {$7$}.
\newblock {\em Proc. London Math. Soc. (3)}, 15:527--542, 1965.

\bibitem[MM82]{Mori1981-82}
S. Mori and S. Mukai.
\newblock Classification of {F}ano {$3$}-folds with {$B\sb{2}\geq 2$}.
\newblock {\em Manuscripta Math.}, 36(2):147--162, 1981/82.
\newblock Erratum: {M}anuscripta {M}ath. 110 (2003), 407.

\bibitem[Mor88]{Mori-1988}
S. Mori.
\newblock Flip theorem and the existence of minimal models for {$3$}-folds.
\newblock {\em J. Amer. Math. Soc.}, 1(1):117--253, 1988.

\bibitem[MP09]{Mori-Prokhorov-2008d}
S.~Mori and Yu. Prokhorov.
\newblock Multiple fibers of del {P}ezzo fibrations.
\newblock {\em Proc. Steklov Inst. Math.}, 264(1):131--145, 2009.

\bibitem[Muk88a]{Mukai-1988}
S. Mukai.
\newblock Curves, {$K3$} surfaces and {F}ano {$3$}-folds of genus {$\leq 10$}.
\newblock In {\em Algebraic geometry and commutative algebra, Vol.\ I}, 
 357--377. Kinokuniya, Tokyo, 1988.

\bibitem[Muk88b]{Mukai1988}
S. Mukai.
\newblock {Finite groups of automorphisms of {K}3 surfaces and the {M}athieu
 group.}
\newblock {\em Invent. Math.}, 94(1):183--221, 1988.

\bibitem[Muk92]{Mukai1992b}
S. Mukai.
\newblock {Curves and symmetric spaces.}
\newblock {\em Proc. Japan Acad., Ser. A}, 68(1):7--10, 1992.

\bibitem[Muk95]{mukai-1995-1}
S. Mukai.
\newblock {Curves and symmetric spaces. I.}
\newblock {\em Am. J. Math.}, 117(6):1627--1644, 1995.


\bibitem[Nam97]{Namikawa-1997}
Y. Namikawa.
\newblock Smoothing {F}ano {$3$}-folds.
\newblock {\em J. Algebraic Geom.}, 6(2):307--324, 1997.

\bibitem[PCS05]{Przhiyalkovskij-Chel'tsov-Shramov-2005en}
V.~V. Przhiyalkovskij, I.~A. Chel'tsov, and K.~A. Shramov.
\newblock {Hyperelliptic and trigonal Fano threefolds.}
\newblock {\em Izv. Math.}, 69(2):365--421, 2005.

\bibitem[Put82]{Puts1982}
P.~J. Puts.
\newblock On some {F}ano-threefolds that are sections of {G}rassmannians.
\newblock {\em Nederl. Akad. Wetensch. Indag. Math.}, 85(1):77--90, 1982.

\bibitem[Rei87]{Reid-YPG1987}
M. Reid.
\newblock Young person's guide to canonical singularities.
\newblock In {\em Algebraic geometry, Bowdoin, 1985 (Brunswick, Maine, 1985)},
{\em Proc. Sympos. Pure Math.}, ~\textbf{46}, 345--414. Amer. Math.
 Soc., Providence, RI, 1987.

\bibitem[Ser09]{Serre2009}
J.-P. Serre.
\newblock A {M}inkowski-style bound for the orders of the finite subgroups of
  the {C}remona group of rank 2 over an arbitrary field.
\newblock {\em Mosc. Math. J.}, 9(1):193--208,  2009.

\bibitem[Shi89]{Shin1989}  
K.-H. Shin.
\newblock {$3$}-dimensional {F}ano varieties with canonical singularities.
\newblock {\em Tokyo J. Math.}, 12(2):375--385, 1989.

\bibitem[Sho93]{Shokurov-1992-e-ba}
V.~V. Shokurov.
\newblock 3-fold log flips.
\newblock {\em Russ. Acad. Sci., Izv., Math.}, 40(1):95--202, 1993.

\bibitem[ST54]{Shephard-Todd}
G.~C. Shephard and J.~A. Todd.
\newblock Finite unitary reflection groups.
\newblock {\em Canadian J. Math.}, 6:274--304, 1954.

\bibitem[Zha06]{Zhang-Qi-2006}
Q.~Zhang.
\newblock Rational connectedness of log {$\mathbf{Q}$}-{F}ano varieties.
\newblock {\em J. Reine Angew. Math.}, 590:131--142, 2006.

\end{thebibliography}
%\bibliographystyle{alpha}%{plain}%{alpha}%%{unrst} {abbrv}
%\end{document} 

\newcommand{\etalchar}[1]{$^{#1}$}
\def\cprime{$'$}
 \def\mathbb#1{\mathbf#1}\def\polhk#1{\setbox0=\hbox{#1}{\ooalign{\hidewidth
 \lower1.5ex\hbox{`}\hidewidth\crcr\unhbox0}}}

\end{document}